\documentclass[12pt]{article}
\usepackage{amsmath,amsthm,amssymb,amsfonts,mathrsfs,multicol}
\usepackage[hypertex, bookmarks, colorlinks=true, plainpages=false, citecolor=green, urlcolor=blue, filecolor=blue]{hyperref}
\newtheorem{theorem}{Theorem}[section]
\newtheorem{lemma}[theorem]{Lemma}
\newtheorem{proposition}[theorem]{Proposition}
\newtheorem{corollary}[theorem]{Corollary}
\newtheorem{definition}[theorem]{Definition}
\newtheorem{example}[theorem]{Example}

\def\a{\bar{a}}\def\b{\bar{b}}
\def\x{\bar{x}}\def\y{\bar{y}}\def\z{\bar{z}}

\def\Nn{\mathbb N}

\def\Rn{\mathbb R}

\def\0{\sf 0}

\makeatletter
\def\dotminussym#1#2{%
  \setbox0=\hbox{$\m@th#1-$}%
  \kern.5\wd0%
  \hbox to 0pt{\hss\hbox{$\m@th#1-$}\hss}%
  \raise.8\ht0\hbox to 0pt{\hss$\m@th#1.$\hss}%
  \kern.5\wd0}
\newcommand{\dotminus}{\mathbin{\mathpalette\dotminussym{}}}
\makeatletter
\newcommand*{\defeq}{\mathrel{\rlap{%
                     \raisebox{0.3ex}{$\m@th\cdot$}}%
                     \raisebox{-0.3ex}{$\m@th\cdot$}}%
                     =}

\begin{document}
\begin{center}
{\Large\sc{On models of affine arithmetic}}\vspace{6mm}

{\bf Seyed-Mohammad Bagheri
\footnote{This work is based upon research funded by Iran National Science Foundation (INSF) under project no.4049526.}}
\vspace{4mm}

{\footnotesize {\footnotesize Department of Pure Mathematics, 
Tarbiat Modares University,\\ Tehran, Iran, P.O. Box 14115-134}}
\vspace{1mm}

bagheri@modares.ac.ir 
\end{center}

\begin{abstract}
By affine arithmetic is meant the set of affine consequences of Peano arithmetic.
This is a continuous theory which is studied in the framework of affine logic, a sublogic of continuous logic.
Affine arithmetic is undecidable. Also, its models are generally lattice ordered
and carry a nontrivial metric. Classical models are then characterized as those which are linearly ordered.
In this paper, the affine variants of several classical results in Peano arithmetic are proved.
In particular, an affine form of Gaifman's splitting theorem is proved.
\end{abstract}
\vspace{2mm}

{\sc Keywords}: {\small Affine logic, Peano arithmetic, cut, cofinal extension}

{\small {\sc AMS subject classification: 03C62, 03C66, 03H15}}
\bigskip
\vspace{2mm}

Classical number theory can be roughly described as the second-order theory of natural numbers.
It is an absolutely categorical theory and $\Nn$ is its unique model.
The first-order theory of $\Nn$ has in contrast non-standard models.
Much of the strength of a theory is due to the expressive power of its ambient logical framework.
An interesting example in this respect is Heyting arithmetic formalized in the framework of intuitionistic logic.
In the present paper, we study a variant of Peano arithmetic formalized in the framework of affine logic.
This is a fragment of continuous logic. Continuous logic is itself an extension of first-order logic.
We concentrate on those properties of numbers which can be expressed by affine formulas.
These formulas are built from atomic formulas using addition-subtraction as connectives
and $\sup$-$\inf$ (the real variants of $\exists$-$\forall$) as quantifiers.
Conjunction, disjunction and implication are not allowed.
Surprisingly, affine formulas have the same expressive power on first-order structures
as first-order formulas.
The main difference between affine arithmetic and first-order arithmetic is that now we deal
with a more general type of models (which we call affine models).
They are metric spaces equipped with the operations of addition and multiplication
as well as the lattice operations. All these operations are continuous.
An affine form of induction holds in these models.
First order models are then exactly those models of affine arithmetic which are
linearly ordered (equivalently, have a $\{0,1\}$-valued metric).
We show that many classical notions and results in Peano arithmetic are generalizable to affine arithmetic.
In particular, overspill/underspill theorems, B\'ezout's theorem and G\"{o}del's lemma hold there.
At the end, we will give an affine variant of Gaifman's splitting theorem.

\section{Affine logic}
Affine logic (AL) is the fragment of continuous logic obtained by restricting
connectives to addition and scalar multiplication.
Let $L$ be a first-order language. To every function symbol $F$ (resp. relation symbol $R$)
is assigned a Lipschitz constant $\lambda_F\geqslant0$ (resp. $\lambda_R\geqslant0$).
Terms are defined as in first-order logic. Affine formulas are inductively defined by
$$1,\ \ \ R(t_1,...,t_n),\ \ \ \phi+\psi,\ \ \ r\phi,\ \ \ \sup_x\phi,\ \ \ \inf_x\phi$$
where $R$ is any $n$-ary relation symbol, $t_1,...,t_n$ are terms and $r\in\Rn$.
The metric symbol $d$ is always considered as a binary relation symbol with $\lambda_d=1$. So, $d(t_1,t_2)$ is a formula.
In full continuous logic (CL) (see \cite{BBHU}), one further allows $\phi\wedge\psi$ and $\phi\vee\psi$ as formulas.
Then, identifying $\neg\phi$ with $1-\phi$,\ $\exists$ with $\sup$ and $\forall$ with $\inf$,
every first-order formula can be regarded as a CL-formula.
Continuous logic has exactly the same expressive power on first-order structures as first-order logic (FL).
In particular, notions such as elementary substructure, definable set etc. are the same (see \cite{BBHU}).
In brief, CL is a conservative extension of FL designed for the study of more general metric structures
and AL is a relatively weak fragment of CL.
In AL, nontrivial finite structures have infinite elementary extensions.
We will however see that AL is sufficiently strong to express many notions in arithmetic.

In the present text, unless otherwise stated, by formula we mean affine formula.
Model theory in AL is very similar to the model theory in CL (see \cite{Bagheri-Lip, ACMT}).
An $L$-structure (or model) $M$ is defined as in CL.
In particular, it is a metric space of diameter at most $1$.
Also, for $n$-ary $F$ and $n$-ary $R$, the maps
$$F^M:M^n\rightarrow M,\ \ \ \ \ \ \ \ \ R^M:M^n\rightarrow[0,1]$$
are $\lambda_F$-Lipschitz and $\lambda_R$-Lipschitz respectively, where
$$d(\x,\y)=\sum_{i=1}^nd(x_i,y_i).$$

For every $L$-formula $\phi(\x)$, model $M$ and $\a\in M$,\ $\phi^M(\a)$ is defined
by induction on the complexity of $\phi$.
One deduces that for every formula $\phi(\x)$, there are $\lambda_{\phi}$, $\mathsf{b}_{\phi}$
depending only on $\phi$ such that
$$\phi^M:M^n\rightarrow[-\mathsf{b}_{\phi},\mathsf{b}_{\phi}]$$ is $\lambda_\phi$-Lipschitz.
A \emph{condition} is an expression of the form $\phi\leqslant\psi$.
A theory is a set $T$ of closed (without free variables) conditions.
Notions such as $M\vDash\phi\leqslant\psi$,\ \ $M\vDash T$,\ \ $T\vDash\phi\leqslant\psi$ etc are all defined in the usual way.
An elementary embedding is a function $f:M\rightarrow N$ such that $\phi^M(\a)=\phi^N(f(\a))$ for every $\phi$ and $\a\in M$.
If $f$ is the identity map, one writes $M\preccurlyeq N$.
An affine form of Tarski's test holds as in CL.
For any structure $M$, there is a metric completion $\bar M$ such that $M\preccurlyeq\bar M$.
This is an important fact, as some constructions produce incomplete structures.
We may however identify every structure with its completion and practically assume that all structures are complete.

The affine counterpart of ultraproduct construction is the ultramean construction.
Let $M_\omega$ be an $L$-structure for each $\omega\in\Omega$ and $\mu$ be a finitely additive
probability measure defined on the power set of $\Omega$ (also called an ultracharge).
Define a pseudo-metric on $\prod_{\omega\in\Omega}M_\omega$ by setting $$d(a,b)=\int_{\omega\in\Omega} d(a_\omega,b_\omega)d\mu.$$
Denote the resulting metric space by $(N,d)$ and the monad of $(a_\omega)$ by $[a_\omega]$.
Then, $N$ is an $L$-structure by setting (e.g. for unary operation and relation symbols):
$$F^N([a_\omega])=[F^{M_\omega}(a_\omega)]$$$$R^N([a_\omega])=\int R^{M_\omega}(a_\omega)d\mu.$$
One proves that for every affine formula $\phi(x_1,...,x_n)$ and $a_k=[a_{k\omega}]$, $k=1,...,n$,
$$\phi^N(a_1,...,a_n)=\int\phi^{M_\omega}(a_{1\omega},...,a_{n\omega})d\mu.$$
This is the affine variant of {\L}o\'{s} theorem. The resulting structure is denoted by $\prod_{\mu}M_\omega$.
It is not hard to see that if $\mu$ is $\sigma$-additive, $\prod_\mu M_\omega$ is complete.
If $M_\omega=M$ for every $\omega$, then $\prod_{\mu}M_\omega$ is denoted by $M^\mu$ and is called the powermean of $M$ by $\mu$.
In this case, $a\mapsto[a]$ is an elementary embedding of $M$ into $M^\mu$ and we may write $M\preccurlyeq M^\mu$.
As an example, let $\Omega=\{0,1\}$ and $\mu(0)=\mu(1)=\frac{1}{2}$.
We may denote $M^\mu$ by $\frac{1}{2}M+\frac{1}{2}M$.

\begin{theorem}
(Affine compactness) Assume for every $0\leqslant\phi_1$, ..., $0\leqslant\phi_n$ in $T$ and
$0\leqslant r_1,...,r_n$, the condition $0\leqslant\sum_ir_i\phi_i$ is satisfiable. Then, $T$ is satisfiable.
\end{theorem}

More generally, if every positive linear combination of conditions in $T$ is satisfied in a class $\mathcal K$ of
$L$-structures, then $T$ has a model of the form $\prod_\mu M_\omega$ where $M_\omega\in\mathcal{K}$ for all $\omega$.

We also need the notions of type and saturation.
An $n$-type for $T$ is a positive linear map $p:D_n(T)\rightarrow\Rn$ such that $p(1)=1$
where $D_n(T)$ is the vector space of $T$-equivalence classes of formulas $\phi(x_1,...,x_n)$.
The set of $n$-types of $T$ forms a compact convex set.
Types over parameters are defined similarly.
A type $p$ is \emph{extreme} if for $r\in(0,1)$,\ \ $p=rp_1+(1-r)p_2$ implies that $p=p_1=p_2$.
Extreme types exist by Krein-Milman theorem \cite{Alfsen}.
A model $M$ is $\kappa$-saturated (resp. extremally $\kappa$-saturated) if for every $A\subseteq M$ with $|A|<\kappa$,
every type (resp. extreme type) $p(x)$ of $Th(M,a)_{a\in A}$ is realized in $M$, i.e. there is $b$ such that
$p(\phi)=\phi^M(b)$ for every $\phi(x)$ with parameters in $A$. Saturated models exist by the usual chain arguments.
We have also the following which will be used later.

\begin{proposition} (see \cite{ACMT, Isomorphism})
Let $L$ be countable and $\mathcal F$ be a countably incomplete ultrafilter on $\Omega$.
Then, $M^{\mathcal F}$ is extremally $\aleph_1$-saturated.
\end{proposition}

Let $M$ be an $L$-structure. A \emph{definable predicate} is a map $P:M^n\rightarrow\Rn$ for which
there is a sequence $\phi_k(\x)$ of formulas (with parameters in $M$) such that $\phi^M_k$ tends uniformly to $P$.
If the parameters used to define $P$ are from $A\subseteq M$, we say it is $A$-definable.
A function $f:M^n\rightarrow M$ is definable if $d(f(\x),y)$ is definable.
Such a function is automatically continuous.
A closed set $D\subseteq M^n$ is definable (resp. simply definable) if $$d(\x,D)=\inf\{d(\x,\y):\ \y\in D\}$$
is definable (resp. equivalent to a formula with parameters).
Definable sets are not closed under Boolean operations.
By an \emph{end-set} we mean a set of the form $$Z(P)=\{\a\in M: P(\a)=0\}$$ where $P:M^n\rightarrow[0,r]$
is definable. A set $X\subseteq M^n$ is \emph{type-definable} if it is the set of common
solutions of a family of conditions $0\leqslant\phi(\x)$. Every definable set is an end-set
and every end-set is type-definable.
In classical structures, first-order definable, affinely definable and end-set are all the same.
A point $b$ is $A$-definable if $\{b\}$ is so. The \emph{definable closure} of $A$ in $M$,
denoted by $\textrm{dcl}_M(A)$, consists of the set of $A$-definable points.
In an extremally $\aleph_0$-saturated model $M$, the maximum of $\phi^M(x,\a)$ is taken by some $b\in M$.
If $b$ is unique, it is definable over $\a$. This is because (in big models) every automorphism fixing $\a$ fixes $b$.
The definable closure of $A\subseteq M$ depends only on the theory of $M$, i.e.
if $A\subseteq M\preccurlyeq N$, then $\textrm{dcl}_M(A)=\textrm{dcl}_N(A)$.
So we may denote it by $\textrm{dcl}(A)$.

\begin{proposition}\label{end-set}
The end-set of a definable predicate $P:M^n\rightarrow[0,r]$ is definable if and only if for every $\epsilon>0$,
there is $\lambda\geqslant0$ such that $$d(\x,Z(P))\leqslant\lambda P(\x)+\epsilon\ \ \ \ \ \ \ \ \forall\x.$$
\end{proposition}

Let us say a model has the \emph{end-set property} if every end-set in it is definable.
First order models and compact models have this property (see \cite{Bagheri-Definability}).
It is also not hard (using the properties of definable sets) to see that if $M$ has the end-set property,
then so has all its extremally $\aleph_0$-saturated elementary extensions.

\begin{proposition} \label{graph}
Let $M$ be $\aleph_0$-saturated. Then, $f:M^n\rightarrow M$ is definable if and only if its graph is a type-definable set.
\end{proposition}

We may extend a bit the definition and call a function $P:D\rightarrow\Rn$, where $D\subseteq M^n$ is definable,
to be definable if there is a sequence of formulas tending to $P$ uniformly on $D$.
Similarly, $f:D\rightarrow M$ is definable if $d(f(\x),y)$ is definable on $D\times M$.
Corresponding variants of the above results then hold similarly.

One way of finding interesting affine theories is affinization.
The \emph{affine part} of a CL-theory $\mathbb T$ is the set of all its affine consequences.
We denote this by $\mathbb{T}_{\textrm{af}}$.

\begin{lemma} \label{ultremean representation}
Let $\mathbb T$ be a CL-theory. Then, $M\vDash\mathbb{T}_{\emph{af}}$ if and only if
it is an elementary substructure of an ultramean of models of $\mathbb T$.
\end{lemma}
\begin{proof}
Assume $M\vDash\mathbb{T}_{\textrm{af}}$.
Let $0\leqslant\phi^M(\a)$ where $\a\in M$. Then, there must exist some $K\vDash\mathbb T$
such that $K\vDash0\leqslant\phi(\a)$. Otherwise, $\mathbb T\vDash\sup_{\x}\phi(\x)\leqslant-\epsilon$
for some $\epsilon>0$ and hence $M\vDash\sup_{\x}\phi(\x)\leqslant-\epsilon$. This is a contradiction.
We conclude that every condition in the (affine) elementary diagram of $M$ is satisfied in a model of $\mathbb T$.
So, $M$ is elementarily embedded in an ultramean of models of $\mathbb T$. The converse is obvious.
\end{proof}


\section{Affine arithmetic} \label{Affine arithmetic}
The order relation is an essential part of the language of Peano arithmetic.
Since logical implication is not allowed in affine logic, we replace $\leqslant$ with the lattice
operations and axiomatize PA and its affine part in the language $$L_A=\{+\ ,\ \cdot\ ,\wedge\ ,\vee,\ 0, 1\}.$$
Here, all operations have Lipschitz constant equal to $1$.
An advantage of this language is that we deal with a unified form of atomic formulas, namely equalities.
This simplifies the process of affinization.
In fact, the lattice operations are more appropriate for axiomatizing the affine part.
The order relation $\leqslant$ is bi-definable with the lattice operations in classical models. 
The strict order relation $<$ is however not expressible in the affine setting.
This is the source of some essential differences between the classical and affine arithmetics.
In the new language, $\textrm{PA}^-$ is axiomatized as the first-order theory of nonnegative parts of lattice ordered
commutative rings with identity, where the resulting order is linear and discrete.
Adding the induction axioms in the usual way, we obtain PA. Now, we go a step further and consider PA as a CL-theory.
So, we replace first-order formulas with more general continuous ones.
Then, the induction schema is stated for all continuous formulas in a new form (since implication is missing).
We will state them below for the affine part of PA. For PA the argument is similar.

A CL-theory $\mathbb T$ has \emph{affine reduction} if every CL-formula is approximated
by affine formulas modulo $\mathbb T$, i.e. for each CL-formula $\phi(\x)$ and $\epsilon>0$,
there is an affine formula $\psi(\x)$ such that $$\mathbb{T}\vDash|\phi(\x)-\psi(\x)|\leqslant\epsilon.$$
Every complete first-order theory has affine reduction (see \cite{Ibarlucia}).
Generally, to prove affine reduction, it is sufficient to approximate quantifier-free CL-formulas by affine formulas.
By the inclusion-exclusion principle, one has that
$$\bigwedge_{i=1}^n\phi_i\ \equiv\ \sum_{\emptyset\neq J\subseteq\{1,...,n\}}(-1)^{|J|+1}\bigvee_{j\in J}\phi_j.$$

\begin{lemma}\label{affine reduction}
Let $\mathbb{T}$ be a first-order theory such that for every pair of atomic formulas $\theta_1,\theta_2$
there is an atomic formula $\eta$ such that $\mathbb{T}\vDash\theta_1\wedge\theta_2\leftrightarrow\eta$.
Then, $\mathbb{T}$ has affine reduction.
\end{lemma}
\begin{proof}
Let $\phi$ be a quantifier-free CL-formula and $r_1<\cdots<r_k$ be its possible range on first order structures.
It is not hard to see that $\phi$ is $\mathbb{T}$-equivalent to a formula of the form
$r_1\phi_1+\cdots+r_k\phi_k$ where every $\phi_i$ is normal disjunctive in the first order sense.
So, we have only to to prove that every quantifier-free first-order $\phi$ is equivalent to an affine formula.
By the normal form theorem, we may assume $\phi=\bigwedge_i\bigvee_j\theta_{ij}$ where $\theta_{ij}$ is atomic or negated atomic.
By the inclusion-exclusion principle, $\phi$ is $\mathbb{T}$-equivalent to a formula of the form
$$\sum_kr_k\phi_k$$ where $r_k=\pm1$ and every $\phi_k$ is a disjunction of atomic or negated atomic formulas.
Using the assumption, every $\phi_k$ can be supposed to be of one of the forms
$$\ \ \ \ \ \theta_1\vee\cdots\vee\theta_\ell,\ \ \ \ \ \neg\eta,\ \ \ \ \ \ \theta_1\vee\cdots\vee\theta_\ell\vee\neg\eta
\ \ \ \ \ \ \ \ \ \ (\theta_i,\eta\ \textrm{atomic}).$$
Then, using the dual of inclusion-exclusion principle and replacing every atomic conjunction with a single one,
we have that $$\mathbb{T}\vDash\phi=\sum_\ell s_\ell\psi_\ell$$
where $s_\ell=\pm1$ and every $\psi_\ell$ is one of the forms
$$\hspace{2cm}\theta,\ \ \ \ \ \neg\eta,\ \ \ \ \ \ \theta\wedge\neg\eta\ \ \ \ \ \ \ \ \ (\theta,\eta\ \textrm{atomic}).$$
Note also that $\theta\wedge\neg\eta$ is equivalent to $\theta-(\theta\wedge\eta)$.
We conclude that every quantifier-free formula is $\mathbb{T}$-equivalent to an affine formula.
\end{proof}

The proof shows also that the affine reduction procedure preserves quantifier ranks.
Another variant of the ultramean construction is direct integrals.
We briefly review this important construction referring the reader to \cite{Ibarlucia} for further details.
Let $M_\Omega=(M_\omega:\omega\in\Omega)$ be a family of $L$-structures and $\mu$ be a probability measure on $\Omega$.
Let $e_I$ be a pointwise enumeration of $M_\Omega$, i.e. a set of elements of $\prod_{\omega\in\Omega}M_\omega$
such that for each $i\in I$, $\{e_i(\omega):\omega\in\Omega\}$ is dense in $M_\omega$.
For simplicity, assume $I$ is countable. Then, $(M_\Omega, e_I)$ is a measurable filed of structures
if for each $n$-tuple $\bar e$ from $e_I$ and $i\in I$, the maps
$$\omega\mapsto P^{M_\omega}(\bar e(\omega)),\ \ \ \ \ \  \ \ \omega\mapsto d^{M_\omega}(F^{M_\omega}(\bar e(\omega)),e_i(\omega))$$
are measurable. A section $f\in\prod_\omega M_\omega$ is measurable if for each $i$, the map
$$\omega\mapsto d^{M_\omega}(f(\omega),e_i(\omega))$$
is measurable. The family of measurable sections up to almost sure equality is denoted by $M_{\Omega,I}$.
Then, like the the ultramean construction, an $L$-structure is defined on $M_{\Omega,I}$ which
is denoted by $\int_\Omega^\oplus M_\omega d\mu(\omega)$. A similar form of the affine {\L}o\'{s}
theorem holds for direct integrals.

As stated before, in affine logic, type spaces of a theory $T$ are compact convex sets.
A model $M\vDash T$ is \emph{extremal} if every type realized in $M$ is extreme.
Such models exist always and form a CL-elementary class in certain cases. The corresponding CL-theory is
then denoted by $T^{\textrm{ex}}$. By a Bauer theory is meant an affine theory for which the type spaces
are Bauer simplices (i.e. they are simplices and the extreme types form closed sets).
We recall two important results from \cite{Ibarlucia} (Theorems 15.6 and 26.1).

\begin{theorem} \label{extremal 1}
Let $\mathbb T$ be a CL-theory in $L$ having affine reduction. Then:

(i) $\mathbb{T}_{\emph{af}}$ is a Bauer theory.

(ii) The extremal models of $\mathbb{T}_{\emph{af}}$ are precisely the models of $\mathbb{T}$.
Hence, $\mathbb{T}\equiv(\mathbb{T}_{\emph{af}})^{\emph{ex}}$.

(iii) Every model of $\mathbb{T}_{\emph{af}}$ is a direct integral of models of $\mathbb{T}$.
\end{theorem}

\begin{theorem} \label{extremal 2}
Let $\mathbb T$ be a first-order theory in $L$ and $M\vDash\mathbb{T}_{\emph{af}}$.
Then the following are equivalent:

(i) $M$ is extremal.

(ii) $M\vDash\mathbb{T}$.

(iii) $M$ is first-order.
\end{theorem}

In particular, $\mathbb T\equiv \mathbb{T}_{\textrm{af}}\cup\textrm{Class}_L$ where $\textrm{Class}_L$
is the CL axioms in $L$ stating that the intended model is first-order (in our context, that $d(x,y)\in\{0,1\}$).
We denote the affine part of PA by AA and the affine part of $\textrm{PA}^-$ by $\textrm{AA}^-$.
Models of AA are called \emph{affine models} of arithmetic. We conclude that:

\begin{theorem}\label{extremals of AA}
PA has affine reduction. So, AA is a Bauer theory and
$\textrm{PA}=(AA)^{\emph{ex}}$. The extremal models of AA
are precisely the models of PA and they are precisely the first-order models of AA.
Every model of AA is a direct integral of models of PA.
\end{theorem}
\begin{proof}
Recall that PA (or even $I\Delta_0$) proves that the pairing function 
$$\langle u,v\rangle=\frac{(u+v+1)(u+v)}{2}+v$$
is bijective. Note that $2\langle u,v\rangle$ is an $L_A$-term.
So, for atomic formulas $f_1=f_2$ and $g_1=g_2$ in $L_A$, we have that
$$\textrm{PA}\vDash((f_1=f_2)\wedge(g_1=g_2))\leftrightarrow2\langle f_1,g_1\rangle=2\langle f_2,g_2\rangle.$$
The other parts are then consequences of Theorems \ref{extremal 1} and \ref{extremal 2}.
\end{proof}

By Theorem \ref{extremals of AA}, AA is neither AL-complete nor finitely axiomatizable.
We will also prove in the next section that it is undecidable in the AL sense.
It is generally not easy to find a concrete axiomatization for the affine part of a CL-theory,
especially when it is undecidable. We will however give a tentative list for the affine part of PA.
All axioms of $\textrm{PA}^-$ except transitivity and discreteness of ordering are already affine,
since $t_1=t_2$ is expressed by the condition $d(t_1,t_2)=0$.
Meanwhile, some first-order logical axioms supporting PA must be replaced with appropriate affine expressions.
In the first step, we have at least the following list of axioms containing the theory of nonnegative parts
of lattice ordered rings (see also \cite{Steinberg}):

\begin{multicols}{2}
R1) Axioms of distributive lattices \cite{Gratzer}

R2) $x+(y+z)=(x+y)+z$

R3) $x+y=y+x$

R4) $x(yz)=(xy)z$

R5) $xy=yx$

R6) $x(y+z)=xy+xz$

R7) $x+0=x$

R8) $x0=0$

R9) $x1=x$






R10) $0\wedge x=0$ 

R11) $x+(y\wedge z)=(x+y)\wedge(x+z)$

R12) $x+(y\vee z)=(x+y)\vee(x+z)$

R13) $x(y\wedge z)=xy\ \wedge\ xz$

R14) $x(y\vee z)=xy\ \vee\ xz$

R15) $x\wedge(yz)\leqslant x(y\wedge z)$

R16) $(x\vee y)^n=x^n\vee y^n$

R17) $(x\wedge y)^n=x^n\wedge y^n$.
\end{multicols}

It is then proved that $x\wedge y+x\vee y=x+y$.
Lattice operations define a partial order by setting $x\leqslant y$ if $x\wedge y=x$.
By R10, \ $0\leqslant x$ for all $x$.
It is also proved that $x\leqslant y$ implies both $x+z\leqslant y+z$ and $xz\leqslant yz$.
In particular, $m\cdot1\leqslant n\cdot1$ whenever $m\leqslant n$.
By setting $z=y+1$ in R15, one deduces that $x\wedge y\leqslant xy$. In particular, $x\leqslant x^2$
and hence $x^2=x$ whenever $x\leqslant1$. In some sense, this is the lattice variant of order discreteness.
The proof of the following lemma is routine.

\begin{lemma} (i) $x\wedge(y+z)\leqslant x\wedge y+x\wedge z$

(ii) $x\vee(y+z)\leqslant x\vee y+x\vee z$

(iii) $x\wedge(yz)\leqslant(x\wedge y)(x\wedge z)$

(iv) $x\vee(yz)\leqslant(x\vee y)(x\vee z)$

(v) $(x\vee y)(x\wedge y)=xy$
\end{lemma}





Let $|x|=d(x,0)$. Then, $\textrm{AA}^-$ contains at least the axioms A1-A7 below (called AA$^\sim$).
\vspace{1mm}

{\sf Axioms of} AA$^\sim$
\vspace{1mm}

A1) Axioms R1-R17    


A2) $|1|=1$

A3) $|x|\leqslant|x\vee y|$

A4) $d(x+y,x+z)=d(y,z)$\ \hspace{37mm} (invariance)


A5) $d(xy,xz)\leqslant d(y,z)\leqslant d(xy,xz)+1-|x|$  


A6) $d(z,x)+d(z,y)=d(z,x\wedge y)+d(z,x\vee y)$\hspace{6mm} (modularity)



A7) $\inf_{z}d((x\wedge y)+z,y)=0$\hspace{38mm} (subtraction).
\bigskip


\begin{lemma} \label{subtraction} In any model of $\textrm{AA}^\sim$,\ \ if $x\leqslant y$, there is a unique $z$ such that $x+z=y$.
\end{lemma}
\begin{proof} (i) By A7, for each $n$, there is $z_n$ such that $d(x+z_n,y)\leqslant\frac{1}{n}$.
Then, $z_n$ is Cauchy and tends to say $z\in M$. Moreover, $x+z=y$. The uniqueness of $z$ is obvious.
\end{proof}

The unique $z$ given by Lemma \ref{subtraction}  is denoted by $y-x$.
More generally, we set $y\dotminus x=y-(y\wedge x)$. So, $y\dotminus x$ is the unique $z$ such that
$y\wedge x+z=y$.

\begin{lemma}
(i) If $x\wedge y=0$, then $d(x,y)=|x|+|y|$.

(ii) $d(x,y)=d(x\dotminus y, y\dotminus x)=|x\dotminus y|+|y\dotminus x|$.
\end{lemma}
\begin{proof} (ii) is a consequence of (i). For (i) note that
$$d(x,y)=d(x\vee y,x\wedge y)=|x\vee y|=|x+y|\leqslant|x|+|y|=|x\vee y|+|x\wedge y|=|x\vee y|.$$
\end{proof}

Some other consequences of $\textrm{AA}^{\sim}$ are as follow.

\begin{lemma}\label{consequences}
(i) $d(x,x\vee y)=d(y,x\wedge y)$

(ii) $d(x,y)=d(x\vee y,x\wedge y)$

(iii) $d(x,y)=d(x,x\wedge y)+d(y,x\wedge y)=d(x,x\vee y)+d(y,x\vee y)$

(iv) $d(x^ny,x^nz)=d(xy,xz)$ for $n\geqslant1$.


(v) If $|x|=1$, then $d(xy,xz)=d(y,z)$. In particular, $d(nx,ny)=d(x,y)$ for $n\geqslant1$.

(vi) $d(x^n,y^n)=d(x,y)$

(vii) $|x\wedge y|+|x\vee y|=|x|+|y|$

(viii) $|yz|=|y\wedge z|$

(ix) $|x+y|=|x\vee y|$



(x) $|x+y_1\cdots y_n|=|x+(y_1\wedge\ldots\wedge y_n)|$.

(xi) $|x(y_1+\ldots+y_n)|=|x(y_1\vee\ldots\vee y_n)|$.

\end{lemma}
\begin{proof}
(i): We have that $$d(x,x\vee y)=d(x+x\wedge y,x\vee y+x\wedge y)=d(x+x\wedge y,x+y)=d(x\wedge y,y).$$

(ii): $$d(x\vee y,x\wedge y)+d(x\wedge y,x\wedge y)=(x,x\wedge y)+d(y,x\wedge y)$$
$$=d(x,x\wedge y)+d(x,x\vee y)=d(x,x)+d(x,y)=d(x,y).$$

(iii): Similar.


(iv): By (ii), we may assume $y\leqslant z$. Let $y+t=z$. Then, by A5 and (v),
$$d(x^ny,x^nz)\leqslant d(xy,xz)=d(0,xt)=d(0,x^nt^n)=d(x^ny,x^nz).$$

(v): Use A5.

(vi): We may assume $x+t=y$. Then, $y^n=x^n+nx^{n-1}t+\cdots+t^n$ and
$$d(x^n,y^n)=d(x^n, x^n+t\cdot *)=d(0,t\cdot *)\leqslant d(0,t)=d(x,y).$$
Conversely, since $t\leqslant t^n$,
$$d(x,y)=d(0,t)\leqslant d(0,t^n+\cdots+ntx^{n-1})=d(x^n,(t+x)^n)=d(x^n,y^n)$$



(vii): Put $z=0$ in A6.

(viii): In R15, set $x=yz$.
Then, since $x\wedge y\leqslant yz$, $$|yz|=|x\wedge yz|\leqslant|x(y\wedge z)|\leqslant|y\wedge z|\leqslant|yz|.$$

(ix): Since $x\vee y\leqslant x+y$, we have that
$$|x+y|\leqslant|x\vee y+x\vee y|=|2(x\vee y)|\leqslant|x\vee y|\leqslant|x+y|.$$

(x): First assume $x=0$. Then, using (iv) and induction, for $k\geqslant3$ we have that
$$|y_1\cdots y_k|=|(y_1y_2)\cdots (y_1y_k)|=|(y_1y_2)\wedge\ldots\wedge(y_1y_k)|$$
$$=|y_1(y_2\wedge\ldots\wedge y_k)|=|(y_1\wedge\ldots\wedge y_k|.$$
Now assume $x$ is arbitrary and consider (for simplicity) the case $n=2$.
Then, $$|x+y_1y_2|=|x\vee y_1y_2|=|x|+|y_1y_2|-|x\wedge y_1y_2|=|x|+|y_1\wedge y_2|-|x\wedge y_1\wedge y_2|$$
$$=|x\vee (y_1\wedge y_2)|=|x+(y_1\wedge y_2)|.$$

(xi): Similar by using (x). 
\end{proof}

\begin{lemma}\label{Boolean lemma}
In any model $M\vDash$ AA$^\sim$ and $x,y\in M$, the following hold:

(i) $|x\wedge1|=|x|$.

(ii) If $x,y\leqslant1$, then $x\wedge y=xy$.

(iii) If $x+y=1$, then $x\wedge y=0$, $x\vee y=1$ and $|x|+|y|=1$.
\end{lemma}
\begin{proof}
(i) $|x\wedge1|=1+|x|-|x\vee1|=|x|$.

(ii) We have that $xy\leqslant x$ and $xy\leqslant y$. Hence, $x\wedge y\leqslant xy\leqslant x\wedge y$.

(iioi) By (iii), $x^2=x$. So, since $x(x+y)=x$, we must have that $xy=0$.
Moreover, $x\vee y+0=x+y=1$. So, $|x|+|y|=|x\wedge y|+|x\vee y|=1$.
\end{proof}


An element $x\in M\vDash$ AA$^\sim$ is \emph{idempotent} if $x^2=x$. It is \emph{normal} if $|x|=1$.

\begin{lemma} \label{normality}
Let $M\vDash\textrm{AA}^\sim$ and $x\in M$. Then,

(i) $x$ is idempotent if and only if $x\leqslant1$.

(ii) $x$ is normal if and only if $1\leqslant x$. In particular, if $x<1$ then $|x|<1$.
\end{lemma}
\begin{proof}
(i) Assume $x\leqslant1$. Then, $x=x\wedge x=x^2$.
Conversely, if $x^2=x$, let $x=y+u$ where $u=x\wedge1\leqslant1$. Then, $u^2=u$ and
$$y^2+2yu+u^2=y+u\leqslant y^2+u^2.$$ So, $yu=0$ and hence $y^2=y$. On the other hand,
$$0=yu=yx\wedge y=y(y+u)\wedge y=y\wedge y=y.$$
We conclude that $x\leqslant1$.

(ii) Let $x\in M$ be normal and choose $z$ such that $x=x\wedge1+z$.
Then, since $z\leqslant x$, $$d(x\wedge z+1,x)=d(z+1,x\wedge1+z)=d(1,x\wedge1)=1-|x|=0.$$
Therefore, $x=x\wedge z+1\geqslant1$. The converse is a consequence of Lemma \ref{consequences} (xi).
\end{proof}

\begin{lemma}
The set of idempotent elements in any model of AA$^\sim$ is a probability algebra with $x'=1-x$ and $\mu(x)=|x|$.
\end{lemma}
\begin{proof}
By Lemma \ref{Boolean lemma}, $x\wedge x'=0$ and $x\vee x'=1$. We must also prove that $(x\vee y)'=x'\wedge y'$.
For this purpose, we need to show that $x'\wedge y'+x\vee y=1$ or that $$(1-x)(1-y)+x\vee y+x+y=1+x+y.$$
This is easily verified by replacing $x$ in the left hand side with $x(1-y)+xy$.
For the last part, note that if $x\wedge y=0$ then $\mu(x\vee y)+0=\mu(x)+\mu(y)$.
\end{proof}

\begin{proposition}\label{models AA sim}
Let $M$ be a model of $\textrm{AA}^\sim$. The following are equivalent:

(i) $(M,\leqslant)$ is linearly ordered.

(ii) There is no $x$ such that $0<x<1$.

(iii) For every $x$,\ \ $x=0$ or $|x|=1$.

(iv) $M$ is a model of $\textrm{PA}^-$.

(v) For all $x,y\in M$, \ $|xy|=|x||y|$.
\end{proposition}
\begin{proof}
(i)$\Rightarrow$(ii): Assume $0<x<1$ and take $y$ such that $x+y=1$.
Then, $x\wedge y=0$ which is impossible by the linearity of ordering.

(ii)$\Rightarrow$(i): 
For every $u$, \ $0\leqslant(u\wedge1)\leqslant1$.
So, $u\wedge1=0$ or $u\wedge1=1$. In the first case, we have that $|u|=|u\wedge1|=0$ and hence $u=0$.
In the second case, $1\leqslant u$.
Therefore, for every $u$, either $u=0$ or $1\leqslant u$.
Now, let $x,y$ be given. By A7, there are $t,s$ such that
$$(x\wedge y)+s=x, \ \ \ \ \ \ \ x\wedge y+t=y.$$
If $s=0$ or $t=0$, we are done. Otherwise, $1\leqslant s\wedge t$.
On the other hand, $$(x\wedge y)+(s\wedge t)=((x\wedge y)+s)\wedge((x\wedge y)+t)=x\wedge y$$
which implies that $s\wedge t=0$. This is a contradiction.

(ii)$\Leftrightarrow$(iii): Suppose $0<|x|<1$. Then $0<|x\wedge1|<1$ and so $0<x\wedge1<1$.
The converse is by Lemma \ref{normality} (ii).

(iii)$\Rightarrow$(iv): $d(x,y)=1$ for every distinct $x,y$.
Moreover, $M$ is linearly ordered. So, $M$ is a model of $\textrm{PA}^-$.

(iv)$\Rightarrow$(i): Obvious.

(iv)$\Rightarrow$(v): Obvious. \ (v)$\Rightarrow$(iii): For every $x$, $|x|=|x^2|=|x|^2$. So, $x=0$ or $|x|=1$.
\end{proof}



\begin{corollary}\label{models PA}
$M\vDash\textrm{AA}$ is a model of PA if and only if it is linearly ordered.
\end{corollary}
\begin{proof}
If $M\vDash\textrm{AA}$ is linearly ordered, it is first order.
Hence, it is a model of PA by Theorem \ref{extremal 2}. The converse is obvious.
\end{proof}

\begin{lemma}
(i) $|n\cdot1|=1$ for every $n\geqslant1$.

(ii) If $x\leqslant y$, then $|x|\leqslant|y|$.

(iii) If $|x|=1$, then $d(y,z)=d(xy,xz)$.
\end{lemma}
\begin{proof}
(i) Use A6.

(ii) Use A5 and A7.

(iii) Use A4 and A5.
\end{proof}

In particular, $n\mapsto n\cdot1$ is an embedding of $\Nn$ into any model of $\textrm{AA}^\sim$.

Let $M$ be a model of $\textrm{AA}^\sim$.
By the continuity of $\wedge$, the set $D_{\leqslant}=\{(x,y):\ x\leqslant y\}$ is closed. Moreover,
$$d((a,b),D_{\leqslant})=\inf_{xy}d((a,b),(x\wedge y, y))=\inf_{xy}[d(a,x\wedge y)+d(b,y)]$$
which shows that $D_{\leqslant}$ is definable.
Also, let $$[a,\infty)=\{u:\ a\leqslant u\},\ \ \ \ \ \ \ \ \ [a,b]=\{u: a\leqslant u\leqslant b\}.$$
These sets are closed and for $a\leqslant b$
$$d(x, [a,\infty))=\inf_t d(x, a\vee t),\ \ \ \ \ d(x,[a,b])=\inf_t d(x,(a\vee t)\wedge b).$$
So, they are definable.

The induction axiom for an affine formula $\phi$ depends on its bound.
For the special case where $$\phi(x,\y)=d(t_1(x,\y),t_2(x,\y))$$
the following condition holds in any model of PA, hence it belongs to AA:
$$\sup_x\phi(x,\y)\leqslant\phi(0,\y)+\sup_x(\phi(x+1,\y)-\phi(x,\y)).$$
More generally, let $\phi(x,\y)$ be an affine formula. Then, there are $r_1<\cdots<r_k$
such that the range of $\phi^M(x,\y)$ is included in $\{r_1,...,r_k\}$ for every first-order $M$.
Also, recall that $\{x:\phi^M(x,\a)=r_i\}$ is first-order definable for each $i$ and $\a\in M$.

\begin{lemma} \label{induction constant}
For every formula $\phi(x,\y)$ there is $\alpha_\phi\geqslant0$ such that the following
condition holds in every model of PA:
$$\sup_x\phi(x,\y)\leqslant\phi(0,\y)+\alpha_\phi\cdot\sup_x(\phi(x+1,\y)-\phi(x,\y)).$$
\end{lemma}
\begin{proof}
Let $r_1,...,r_k$ be as above. If $k=1$, set $\alpha_\phi=0$. Otherwise,
$\alpha_\phi=\frac{2\mathsf{b}_\phi}{r}$
where $$r=\min\{r_{i+1}-r_i:\ 1\leqslant i<k\}.$$
It is clear that for each $M\vDash\textrm{PA}$
$$0\leqslant\Delta_M(\y)=\sup_{x}(\phi^M(x+1,\y)-\phi^M(x,\y)) \ \ \ \ \ \ \ \ \forall\y.$$
Then, for any $\y\in M$, the intended condition holds by induction axioms if $\Delta_M(\y)=0$.
Otherwise, one has that $\Delta_M(\y)\geqslant r$ in which case
$$\sup_x\phi^M(x,\y)-\phi^M(0,\y)\leqslant2\mathsf{b}_\phi\leqslant\frac{2\mathsf{b}_\phi}{r}\Delta_M(\y).$$
\end{proof}

We conclude that the following axiom schemas are part of AA.
\vspace{2mm}

A11)\ \ $\sup_x\phi(x,\y)\leqslant\phi(0,\y)+\alpha_\phi\cdot\sup_x(\phi(x+1,\y)-\phi(x,\y))$.

A12)\ \ $\sup_x\phi(x,\y)\leqslant\phi(0,\y)+\alpha_\phi\cdot\sup_{xu}(\phi(x+(u\wedge1),\y)-\phi(x,\y))$.
\bigskip

Note that A11 implies A12. The coefficient $\alpha_\phi$ is the \emph{induction constant} of $\phi$.
We leave open the question of whether every definable predicate has an appropriate induction constant.

\begin{corollary}\label{induction applications}
Assume $M\vDash\textrm{AA}^\sim+\textrm{A11}$.

(i) If $\phi(0)\leqslant r$ and $\phi(x+1)\leqslant\phi(x)$ for every $x$, then $\phi(x)\leqslant r$ for every $x$.

(ii) If $\phi(0)=r$ and $\phi(x+1)=\phi(x)$ for every $x$, then $\phi(x)=r$ for every $x$.
\end{corollary}

In the next sections, we will frequently apply the technique used above.
Instead of deducing results from the axioms, we write affine
conditions holding in all models of PA. We then deduce that they belong to AA.
This is also an easy way of manifesting the power of affine expressions in arithmetic.
Note that the representation part of Theorem \ref{extremals of AA} can provide shorter and easier proofs for some of these results.

As stated above, every model of AA is a direct integral of models of PA.
Easier models are obtained by the ultramean construction.
If $\mu$ is an ultracharge on a set $\Omega$ and $M_\omega$ is a model of PA for each $\omega\in\Omega$, then $\prod_\mu M_\omega$
is a model of AA. In particular, $\Nn^\mu$ is a model of AA.

\section{Undecidability}
Usually, G\"{o}del numbering is performed for first-order formulas \cite{Boolos, Kaye}.
This can be however done for integral (i.e. with integer coefficients) continuous $L_A$-formulas similarly.
In particular, we may arrange that $\mathsf{b}_\phi$, $\lambda_\phi$ be integer and the functions
$\phi\mapsto\mathsf{b}_\phi$ and $\phi\mapsto\lambda_\phi$ be computable.
Assuming such a numbering is given, we denote the G\"{o}del number of $\phi$ by $\lceil\phi\rceil$ (an injective map).
Then, for a CL-theory $\mathbb T$ in $L_A$ set:
$$\mathcal{G}(\mathbb{T})=\{\lceil\phi\rceil:\ \mathbb{T}\vDash0\leqslant\phi,\ \ \phi\ \textrm{is an integral CL-formula}\}$$
$$\mathcal{G}_{\textrm{af}}(\mathbb{T})=\{\lceil\phi\rceil:\ \mathbb{T}\vDash0\leqslant\phi,\ \ \phi\ \textrm{is an integral AL-formula}\}.$$
We say $\mathbb{T}$ is \emph{decidable} in CL (resp. in AL) if the set
$\mathcal{G}(\mathbb T)$ (resp. $\mathcal{G}_{\textrm{af}}(\mathbb T)$) is computable.
The first definition essentially coincides with the standard definition for first-order theories.
Undecidability of AA is a consequence the proof of Lemma \ref{affine reduction} since the affine reduction
procedure is effective.
Given a first order sentence $\sigma$ (put in the prenex form), let $\sigma^{\textrm{af}}$ be its affine reduction.
Then, $\sigma\mapsto\sigma^{\textrm{af}}$ is computable and PA$\vDash\sigma$ if and only if AA$\vDash1\leqslant\sigma^{\textrm{af}}$.
So, decidability of AA leads to the decidability of PA, which is impossible.
We will however give a more complete description containing also the undefinability of truth.

A function $f:\omega^n\rightarrow\omega$ is \emph{representable} in an affine $L_A$-theory $T$ if
there is an integral affine $\phi(\x,y)$ such that
$$T\vDash d(f(\bar{n}),y)=\phi(\bar{n},y)\ \ \ \ \ \ \ \ \ \forall\bar{n}\in\omega^n.$$
A relation $S\subseteq\omega^n$ is \emph{representable} in $T$ if there is an integral affine $\psi(\x)$ such that
$$\bar{n}\in S\ \ \Longrightarrow\ \ T\vDash1\leqslant\psi(\bar{n})$$
$$\bar{n}\not\in S\ \ \Longrightarrow\ \ T\vDash\psi(\bar{n})\leqslant0.$$
Below, we adopt proofs of the classical results for the present situation.

\begin{lemma}
Every recursive function or recursive relation is representable in AA.
\end{lemma}
\begin{proof}
Let $f$ be a recursive function.
So, there is a first-order formula $\theta(\x,y)$ such that
$$\textrm{PA}\vDash\forall y(f(\bar{n})=y\ \leftrightarrow\ \theta(\bar{n},y))\ \ \ \ \ \ \forall\bar{n}.$$
We may write this as
$$\textrm{PA}\vDash d(f(\bar{n}),y)=1-\theta(\bar{n},y)\ \ \ \ \ \ \forall\bar{n}.$$
By Theorem \ref{extremals of AA} (and its proof), there is an integral affine formula $\phi(\x,y)$ such that
$$\textrm{PA}\vDash1-\theta(\x,y)=\phi(\x,y).$$
We conclude that PA (and hence AA) satisfies
$$d(f(\bar{n}),y)=\phi(\bar{n},y)\ \ \ \ \ \ \forall\bar{n}.$$
For the second part, assume a recursive set $S$ is represented in PA by
the formula $\theta(\x)$. So, one has that
$$\bar{n}\in S\ \Longrightarrow\ \textrm{PA}\vDash\theta(\bar{n})$$
$$\bar{n}\not\in S\ \Longrightarrow\ \textrm{PA}\vDash\neg\theta(\bar{n}).$$
Let $\psi(\x)$ be an integral affine formula such that $\textrm{PA}\vDash\theta(\x)=\psi(\x)$.
Then, PA satisfies $0\leqslant\psi(\x)\leqslant1$ as well as the conditions:
$$1\leqslant\psi(\bar{n})\ \ \ \ \ \ \textrm{for}\ \bar{n}\in S$$
$$\psi(\bar{n})\leqslant0\ \ \ \ \ \ \ \textrm{for}\ \bar{n}\not\in S.$$
We conclude that AA satisfies these conditions.
\end{proof}
\vspace{2mm}

For an integral $\phi$ define $$\textsf{d}(\phi(x))=\inf_x(\lambda_\phi d(x,\lceil\phi\rceil)+\phi(x)).$$
Then, since $\phi(y)\leqslant\lambda_\phi d(x,y)+\phi(x)$ in every model,
we have that $$\vDash\phi(y)=\inf_x(\lambda_\phi d(x,y)+\phi(x)).$$
The following lemma is proved as in the classical case. It states that
$\phi\mapsto\lceil\mathsf{d}\phi\rceil$ is computable.

\begin{lemma}
There is a recursive function $\textsf{dg}:\omega\rightarrow\omega$ such that
$\textsf{dg}(\lceil\phi\rceil)=\lceil\textsf{d}\phi\rceil$ for every $L_A$-formula $\phi$.
\end{lemma}

\begin{lemma} (Fixed point)
Let $\phi(x)$ be an integral affine formula. Then, there is an integral affine sentence $\sigma$
such that $$\textrm{AA}\vDash\sigma=\phi(\lceil\sigma\rceil).$$
\end{lemma}
\begin{proof}
Assume $\textsf{dg}$ is represented by the affine formula $\psi(x,y)$ so that
$$\textrm{AA}\vDash d(y,\textsf{dg}(n))=\psi(n,y)\ \ \ \ \ \ \ \forall n.\hspace{18mm} (\dag)$$
Let $\delta(x)=\inf_y(\lambda_\phi\psi(x,y)+\phi(y))$ and $n=\lceil\delta(x)\rceil$.
Define $$\sigma=\textsf{d}(\delta(x))=\inf_x(\lambda_\delta d(x,n)+\delta(x)).$$
Let $k=\textsf{dg}(n)=\lceil\sigma\rceil$. Then
$$\textrm{AA}\vDash\sigma=\delta(n)=\inf_y(\lambda_\phi\psi(n,y)+\phi(y)).$$
So, by ($\dag$) above, $$\textrm{AA}\vDash\sigma=\inf_y(\lambda_\phi d(y,k)+\phi(y))=\phi(k)=\phi(\lceil\sigma\rceil)$$
as required.
\end{proof}
\vspace{1mm}

Recall that first-order and affine definability are equivalent in first-order structures.

\begin{theorem} (Undefinability of affine truth)
Let $$S=\{\lceil\sigma\rceil:\ 0\leqslant\sigma^{\Nn},\ \ \ \sigma\ \textrm{is integral affine}\}.$$
Then, $S$ is not definable in $\Nn$.
\end{theorem}
\begin{proof}
Assume it is definable. Then, there is an affine integral $\beta(x)$ such that
$$S=\{n\in\Nn: 0\leqslant\beta^{\Nn}(n)\}.$$
Use fixed point lemma for $-\beta(x)-1$.
Then, $\Nn\vDash\sigma=-\beta(\lceil\sigma\rceil)-1$ for some integral affine sentence $\sigma$.
In other words, $\beta^{\Nn}(\lceil\sigma\rceil)=-\sigma-1$. Suppose $0\leqslant\sigma^{\Nn}$.
Then, $\beta^{\Nn}(\lceil\sigma\rceil)<0$.
So, $\lceil\sigma\rceil\not\in S$, i.e. $0\nleqslant\sigma^{\Nn}$. Similarly, if $\sigma^{\Nn}\leqslant-1$
then $\beta^{\Nn}(\lceil\sigma\rceil)\geqslant0$ and hence $0\leqslant\sigma^\Nn$.
\end{proof}

\begin{proposition} \label{undecidability}
AA is undecidable in affine logic.
\end{proposition}
\begin{proof}
Assume that $\mathcal{G}_{\textrm{af}}(\textrm{AA})$ is computable.
Then it is represented by an affine formula $\beta(x)$. Let $\sigma$ be a fixed point of $-\beta(x)$.
So, $$\textrm{AA}\vDash\sigma=-\beta(\lceil\sigma\rceil).$$
Now, if $\textrm{AA}\vDash0\leqslant\sigma$, then $\lceil\sigma\rceil\in\mathcal{G}_{\textrm{af}}(\textrm{AA})$.
So, by representability, $\textrm{AA}\vDash1\leqslant\beta(\lceil\sigma\rceil)$.
So, $\textrm{AA}\vDash\sigma\leqslant-1$. This is a contradiction.
Also, if $\textrm{AA}\nvDash0\leqslant\sigma$, then $\lceil\sigma\rceil\not\in\mathcal{G}_{\textrm{af}}(\textrm{AA})$.
So, by representability, $\textrm{AA}\vDash\beta(\lceil\sigma\rceil)\leqslant0$.
So, $\textrm{AA}\vDash0\leqslant\sigma$, again a contradiction.
\end{proof}

Alternatively, the effective affine reduction in section \ref{Affine arithmetic} yields a short proof of Proposition \ref{undecidability}.

\section{Hierarchy}
By the classical hierarchy theorem, the arithmetic hierarchy of formulas does not collapse modulo PA.
We prove that a similar property holds in the affine case. Set
$$\sup_{x\leqslant t}\phi(x)=\sup_x\phi(x\wedge t)$$
$$\inf_{x\leqslant t}\phi(x)=\inf_x\phi(x\wedge t)$$
where $t$ is any term. Then, $\sup_{x\leqslant t}$ and $\inf_{x\leqslant t}$ are called \emph{bounded quantifiers}.
A formula is \emph{bounded} if all its quantifiers are bounded.
The affine hierarchy of formulas is defined as follows:

- $\Sigma_0=\Pi_0$ is the set of bounded AL-formulas

- $\phi\in\Sigma_{n+1}$ if it is of the form $\sup_{\bar x}\psi$ where $\psi\in\Pi_n$

- $\phi\in\Pi_{n+1}$ if it is of the form $\inf_{\bar x}\psi$ where $\psi\in\Sigma_n$.
\bigskip

If $T$ is an affine theory, we extend a bit the terminology and say that $\phi$ is $\Sigma_n$ (resp. $\phi$ is $\Pi_n$)
if it is $T$-equivalent to a $\Sigma_n$ (resp. $\Pi_n$) formula.
In this case, to emphasize on $T$, we may denote them by $\Sigma_n(T)$ etc. To justify the terminology,
note that we tacitly identify $\sup_x\phi$ with the condition $0\leqslant\sup_x\phi$.
Both $\Sigma_n$ and $\Pi_n$ are closed under addition and multiplication by any $r\geqslant0$.
Also, $\phi\in\Sigma_n$ iff $-\phi\in\Pi_n$.
For a CL-theory $\mathbb{T}$, the classical hierarchy is defined as above except that
AL-formulas are replaced with CL-formulas.  For both hierarchies we have that
$$\Sigma_n\ \subseteq\ \Sigma_{n+1}\cap\Pi_{n+1},\ \ \ \ \ \ \ \Pi_n\ \subseteq\ \Sigma_{n+1}\cap\Pi_{n+1}.$$
It is clear that if $T$ is the affine part of $\mathbb{T}$, then, two affine formulas are $T$-equivalent
if and only if they are $\mathbb{T}$-equivalent.
So, $\Sigma_n(T)\subseteq\Sigma_n(\mathbb{T})$ and $\Pi_n(T)\subseteq\Pi_n(\mathbb{T})$.
Now, assume $\mathbb{T}=\textrm{PA}$ and $T=\textrm{AA}$ (or else, we may assume $\mathbb{T}=Th(\Nn)$
and $T$ is its affine part). Then

\begin{proposition}
For each $n\geqslant1$, $\Sigma_n(T)\nsubseteq\Pi_n(T)$. So, the affine arithmetic hierarchy does not collapse modulo $T$.
\end{proposition}
\begin{proof}
Assume $\Sigma_n(T)\subseteq\Pi_n(T)$. Then, $\Sigma_n(T)=\Pi_n(T)$ and hence $\Sigma_n(T)=\Sigma_{m}(T)$
for all $m\geqslant n$. So, up to $T$-equivalence, $\Sigma_n(T)$ is the set of all affine formulas.
This implies that affine types are separated by $\Sigma_n(T)$-formulas.
In particular, extreme types of $T$ (which are exactly the classical types of $\mathbb T$)
are separated by $\Sigma_n$-formulas.
However, by the classical hierarchy theorem, formulas of $\Sigma_n(\mathbb T)$ can not separate all classical types.
\end{proof}

\section{Collection}\label{collection section}
The classical collection axiom for $\phi$ is the universal closure of the following formula:
$$\forall x<t\exists\y\ \phi(x,\y,\z)\rightarrow\exists s\forall x<t\exists\y<s\ \phi(x,\y,\z).$$
The continuous form of collection axiom for a CL-formula $\phi$ is then the following:
$$\inf_{x\leqslant t}\sup_{\y}\phi(x,\y,\z)\leqslant\sup_s\inf_{x\leqslant t}\sup_{\y\leqslant s}\phi(x,\y,\z).$$
Clearly, it holds in models of PA. Its converse is always true. So, we may write it as
$$\inf_{x\leqslant t}\sup_{\y}\phi(x,\y,\z)=\sup_s\inf_{x\leqslant t}\sup_{\y\leqslant s}\phi(x,\y,\z)$$
and deduce that it belongs to AA (i.e. its universal closures).
Following \cite{Kaye}, we denote this condition by $B_\phi$.
Let $$\textrm{Coll}_n=\textrm{AA}^\sim+\{B_\phi:\ \phi\in\Sigma_n\}$$
$$\textrm{Coll}=\bigcup_n\textrm{Coll}_n.$$
The following lemma is then proved by induction on $n$.

\begin{lemma}\label{coll}
$\Sigma_n({\emph{Coll}}_n)$ and $\Pi_n({\emph{Coll}}_n)$ are closed under bounded quantification.
\end{lemma}

Using a stronger form of the collection axioms we can find arbitrarily large Archimedean models of AA.

\begin{example} \label{examples-bounded}
\em{Let $\mu$ be an ultracharge on the power set of $\Omega$.
Let $$\mathbf{N}=\{a\in\Nn^\mu: \exists k\in\Nn\  i\ \ a_i\leqslant k\ \textrm{for $\mu$-almost all}\ i\}.$$
We prove that $\Nn\preccurlyeq\bar{\mathbf{N}}\preccurlyeq\Nn^\mu$ where $\bar{\mathbf{N}}$ is the topological closure of $\mathbf{N}$.
For this purpose, we use the following form of the collection axioms holding in $\Nn$:
$$\forall t\exists s\forall\x<t[\exists y\phi(\x,y)\rightarrow\exists y<s\ \phi(\x,y)].$$
Using the affine language, we may write it for any affine formula $\phi$ as
$$\sup_t\inf_s\sup_{\x\leqslant t}\hspace{.5mm}[\hspace{.5mm}\inf_{y\leqslant s}\phi(\x,y)-\inf_{y}\phi(\x,y)]\leqslant0.$$
Clearly, it is satisfied in $\Nn$. By Tarski's test, we need to prove that for every such
$\phi(\x,y)$ and $\a\in\mathbf{N}$ one has that
$$\inf\{\phi^{\Nn^\mu}(\a,b)\ | \ b\in\mathbf{N}\}=\inf\{\phi^{\Nn^\mu}(\a,c) \ | \ c\in\Nn^{\mu}\}.$$
For simplicity, assume the length of $\x$ is $1$. In the collection axiom above, put $t=n\in\Nn$.
Then, there is $s=m$ such that for all $x\leqslant n$ in $\Nn$ one has that
$$\inf_{y\leqslant m}\phi^\Nn(x,y)\leqslant\inf_{y}\phi^\Nn(x,y).$$
Take an arbitrary $a\in\mathbf{N}$ such that $a_i\leqslant n$ for almost all $i\in\Omega$.
Then, for almost all $i$, there is $b_i\leqslant m$ such that $$\phi^{\Nn}(a_i,b_i)\leqslant\inf_y\phi^{\Nn}(a_i,y)$$
Let $b=[b_i]$. Then, by integrating, we have that
$$\phi^{\Nn^\mu}(a,b)=\int\phi^{\Nn}(a_i,b_i)d\mu\leqslant
\int\inf_y\phi^{\Nn}(a_i,y)d\mu=\inf_y\phi^{\Nn^\mu}(a,y).$$}
\end{example}

\section{Least number principle}
Let $\phi(x,\y)$ be an affine formula. As stated before, there are $r_1<\cdots<r_k$ such that
the range of $\phi^M(x,\y)$ is included in $\{r_1,...,r_k\}$ for every first-order structure $M$.
In models of PA, the condition
$$\sup_x\phi(x,\y)\leqslant\sup_{x}\ [2\phi(x,\y)-\sup_{t\leqslant x}(\phi(t,\y)+\lambda d(t,x))]
\hspace{14mm}(L^1_x\phi)$$
states that there exists $a$ such that
$$\sup_x\phi(x,\y)\leqslant\phi(a,\y)-\sup_{t\leqslant a}(-\phi(a,\y)+\phi(t,\y)+
\lambda d(t,a))\leqslant\phi(a,\y).\ \ \ \ \ (\star)$$
Equivalently,
$$\phi(a,\y)=\sup_x\phi(x,\y) \ \ \ \ \& \ \ \ \ \phi(t,\y)\leqslant\phi(a,\y)-\lambda d(t,a)
\ \ \ \ \forall t<a.\ \ \ \ \ \ \ (\star\star)$$
By the least number principle in models of PA, such an $a$ exists if
$$\lambda=\min\{r_i-r_{i-1}:\ 1<i\leqslant k\}$$
(or $\lambda=1$ if $k=1$, in which case $a=0$ works).
We conclude that for suitable $\lambda>0$, the condition $(L^1_x\phi)$ holds in all models of AA.
On the other hand, let
$$\psi(z,\y)=\sup_x\phi(x,\y)-[2\phi(z,\y)-\sup_{t\leqslant z}(\phi(t,\y)+\lambda d(t,z))].$$
Then, $0\leqslant\inf_{z\y}\psi^M(z,\y)$ for any $M\vDash\textrm{PA}$.
Since, otherwise, as in $(\star)$ for some $a,\y$ in $M$ one must have that
$$\sup_x\phi^M(x,\y)<\phi^M(a,\y)$$
which is impossible. Furthermore, in such $M$,\ \ $\psi^M(a,\y)=0$ if and only if $(\star\star)$ holds.
If $k=1$, set $\gamma=0$. Otherwise, there is a minimum $\gamma>0$ in the ranges of
$\psi^M(x,\y)$ when $M$ varies among models of PA.
Then, by the uniqueness of such $a$, we conclude that
$$\gamma d(z,u)\leqslant\psi(z,\y)+\psi(u,\y) \hspace{18mm} (L^2_x\phi)$$ holds in all models of PA.
The family of conditions $(L^1_x\phi)$ and $(L^2_x\phi)$ (for all $\phi$) is called the \emph{least number axioms}.
These conditions belong to AA and guarantee the existence and uniqueness of $a$ stated above in models of AA.
We have therefore the following form of the \emph{least number principle} (LNP).

\begin{proposition} \label{least number}
Let $M\vDash\textrm{AA}$ and $\phi(x,\y)$ be a formula.
Then, for each $\b$ there is a unique $a$ such that
$$\phi^M(a,\b)=\sup_x\phi^M(x,\b),\ \ \ \ \ \ \ \ \ \ \ \phi^M(t,\b)<\phi^M(a,\b) \ \ \ \ \ \forall t<a.$$
Moreover, the map $\b\mapsto a$ given above is definable.
\end{proposition}
\begin{proof}
For the existence, take an extremally $\aleph_0$-saturated extension $M\preccurlyeq N$.
Such $a$ exists in $N$ by $(L^1_x\phi)$. Moreover, by $(L^2_x\phi)$, it is unique.
Hence it is in the definable closure of $\b$. We conclude that $a\in M$.
\end{proof}

\begin{corollary} \label{prime model}
Let $A\subseteq M\vDash\textrm{AA}$. Then, $\textrm{dcl}(A)\preccurlyeq M$.
In particular, $\textrm{dcl}(\emptyset)\preccurlyeq N$ for every $N\equiv M$.
\end{corollary}
\begin{proof}
It is clear that $\textrm{dcl}(A)$ is a substructure of $M$. It is also metrically closed.
We use Tarski's test to show that $\textrm{dcl}(A)\preccurlyeq M$.
Assume $\b\in \textrm{dcl}(A)$ and $r=\sup_x\phi^M(x,\b)$. Let $a$ be as in Proposition \ref{least number}.
Then, $a\in dcl(A)$ and $r\leqslant\phi^M(a,\b)$.
For the second part, use the affine elementary amalgamation property to find $\bar M$ with
$M\preccurlyeq\bar M$ and $N\preccurlyeq\bar M$. Then, $\textrm{dcl}(\emptyset)\preccurlyeq\bar M$.
\end{proof}

It is also clear that $dcl(\emptyset)$ in Corollary \ref{prime model} is rigid and has no proper elementary substructure.
In particular, $Th(M)$ has a prime model. Similar results hold for $Th(M,a)_{a\in A}$ whenever
$M\vDash\textrm{AA}$ and $A\subseteq M$. In models of $\textrm{PA}^-$,\ \ LNP and induction axioms are equivalent.
The following proposition states that $\textrm{AA}^\sim+$ LNP implies weak induction axioms.

\begin{proposition}
Let $M\vDash\textrm{AA}^\sim$ and assume that LNP holds in $M$.
Then, for each $\phi(x,\y)$, there is $\alpha\geqslant0$ such that $M$ satisfies the condition:
$$\sup_x\phi(x,\y)\leqslant\phi(0,\y)+\alpha\cdot\sup_{xu}(\phi(x+(u\wedge1),\y)-\phi(x,\y)).$$
\end{proposition}
\begin{proof}
We may assume without loss that $M$ is $\aleph_0$-saturated and $\phi^M(0,\y)=0$.
Let $$\Delta=\sup_{xu\y}(\phi^M(x+(u\wedge1),\y)-\phi^M(x,\y)).$$
Clearly, $\Delta\geqslant0$. We have then two cases:

(i) $\Delta=0$: In this case, we set $\alpha=0$ and we must have that $\sup_x\phi^M(x,\y)\leqslant0$ for every $\y$.
Since, otherwise, for some fixed $\y\in M$,\ \ $\sup_x\phi^M(x,\y)=r>0$.
So, there is $a$ such that $\phi^M(a,\y)=r$ and $\phi^M(t,\y)<r$ whenever $t<a$.
Since $a\neq0,$ we may write $a=c+u$ where $u=a\wedge1$ is a nonzero idempotent. Therefore, $c<a$ and
$$\phi^M(a,\y)=\phi^M(c+u,\y)\leqslant\phi^M(c,\y)<\phi^M(a,\y).$$
This is a contradiction.

(ii) $\Delta>0$: In this case, the claim holds with $\alpha=\frac{2\mathsf{b}_\phi}{\Delta}$.
\end{proof}

\section{Cut}
Proofs in this and the next sections are usually adaptations of the classical proofs
for the affine situation. They are mostly taken from \cite{Kaye}. Let $M$ be a model of AA
(usually $\textrm{AA}^\sim+$ LNP is sufficient).

\begin{lemma}
Let $I\subseteq M$ be nonempty and $u$ be fixed. Assume $x\in I$ implies $x+u\in I$ for every $x$.
Then $d(a+u,I)\leqslant d(a,I)$ for every $a$. The converse is true if $I$ is metrically closed.
\end{lemma}
\begin{proof}
The second claim is easy. For the first part, set $$I+u=\{x+u: x\in I\}.$$
Then, $I+u\subseteq I$ and $d(a+u,I)\leqslant d(a+u,I+u)=d(a,I)$.
\end{proof}

Then, using Corollary \ref{induction applications}, one proves that

\begin{corollary}\label{Definable-induction1}
Let $D\subseteq M\vDash\textrm{AA}$ be simply definable.
Assume $0\in D$ and for all $x$,\ \ $x\in D$ implies $x+1\in D$. Then $D=M$.
\end{corollary}

\begin{definition}
\emph{A \emph{cut} in $M$ is a nonempty set $I\subseteq M$ such that for all $x,y$,
\\ (i) if $x\leqslant y\in I$ then $x\in I$ (ii) if $x\in I$ then $x+1\in I$.
A cut $I$ is \emph{directed} if $x\vee y\in I$ whenever $x,y\in I$.}
\end{definition}

If $I$ is a cut, then $x\in I$ implies $x+u\in I$ for any idempotent $u$.
So, $d(a+u,I)\leqslant d(a,I)$ for all $a$.
If $I<I^c$, then $I$ is directed and moreover, $x\wedge y\in I^c$ whenever $x,y\in I^c$.
The topological closure of a cut is a cut.

\begin{example}\label{examples-cut}
\em{(i) Let $\mu$ be an ultracharge on $\Nn$ for which $\mu(i)>0$ for all $i$ and 
fix $A\subseteq\Nn$. Let
$$I_A=\{[a_i]\in\Nn^\mu:\ \exists r\in\Nn\ \forall i\in A\ \ a_i\leqslant r\}.$$
Then, $I_A$ is a cut. Assume $\mu$ is not an ultrafilter and $\mu(i)=0$ for every $i$.
If $A=\Nn$, then $I=I_A$ is a bounded cut i.e. $I<b$ for some $b$ (take $b_i=i$).
It is however not the case that $I<I^c$ (neither $1\leqslant I^c$).
If $A=2\Nn$ and $\mu(A)=\frac{1}{2}$, then $I_A$ is proper but not bounded.
Also, in this case, if $B=\Nn-A$, then $I_A\cup I_B$ is a cut which is not directed.
\bigskip

\noindent(ii) In example (i), let $I=I_\Nn$ and assume $\mu(i)=2^{-i-1}$. Let $a_k=(a_{ki})$
where $a_{ki}=i$ if $i\leqslant k$ and $a_{ki}=k$ if $i>k$.
Then, $a_k\in I$ converges to a point outside $I$. So, $I$ is not topologically closed.
\bigskip

\noindent(iii) Let $\mu$ be as in example (ii) and $K\vDash\textrm{PA}$ be nonstandard.
Let $$I=\{[a_i]\in K^\mu:\ \exists i\ a_i\in\Nn\}.$$
Then, $I$ is a cut in $K^\mu$ and one has that $1\leqslant I^c$.
\bigskip

\noindent(iv) Let $\Nn[x]$ be the set of polynomials with coefficients in $\Nn$. Then the set
$$I=\{[a_i]\in\Nn^\mu:\ \exists f(x)\in\Nn[x]\ \forall i\ a_i\leqslant f(i)\}\subseteq\Nn^\mu$$
is a bounded cut if $\mu(i)=0$ for every $i$.
\bigskip

(v) If $M$ is not linearly ordered, then $\Nn$ is not an initial segment of $M$.}
\end{example}

\begin{proposition}
Let $I\subseteq M$ be a proper cut.

(i) $I$ is not simply definable.

(ii) If $1\leqslant I^c$, then $I^c$ is not simply definable.
\end{proposition}
\begin{proof}
(i) Assume it is simply definable. Since $I\neq M$, there exists $a$ such that
$$0<\sup_xd(x,I)=d(a,I),\ \ \ \ \ \ \ \ \ \ \ d(t,I)<d(a,I) \ \ \ \ \forall t<a.$$
So, $a\not\in I$ and hence $a=b+u$ for some $b$ and nonzero idempotent $u=a\wedge1$.
Therefore, $b\not\in I$ and $$d(a,I)\leqslant d(b,I)<d(a,I).$$ This is a contradiction.

(ii) By the assumption, for each $i\geqslant1$,\ \  $i\in I^c$ if and only if $i-1\in I^c$.
It is then easily proved that $$d(x,I^c)=d(x+1,I^c)\ \ \ \ \ \ \ \forall x.$$
By the least number axioms, there exists $a$ such that
$$0=\inf_xd(x,I^c)=d(a,I^c),\ \ \ \ \ \ \ \ \ \ \ \ \ \ \ d(a,I^c)<d(t,I^c)\ \ \ \forall t<a.$$
So, $a\in I^c$ and $d(a,I^c)<d(a-1,I^c)\leqslant d(a,I^c)$. This is a contradiction.
\end{proof}

For $D\subseteq M$ and $\epsilon>0$ set $D^\epsilon=\{x\in M:\ d(x,D)<\epsilon\}$.

\begin{proposition} (Overspill)
Let $I$ be a proper cut and $D\subseteq M$ be simply definable. Assume $I\subseteq D$.
Then, for all $\epsilon>0$, there exists $a\in I^c$ such that $[0,a]\subseteq D^\epsilon$.
\end{proposition}
\begin{proof}
Assume not. Then, there is $\epsilon>0$ such that for every $a\in I^c$ one has that
$[0,a]\not\subseteq D^\epsilon$. Set $$P(x)=\sup_{y\leqslant x}d(y,D)=\sup_yd(x\wedge y,D).$$
Then, $P(x)$ is a formula and $I=Z(P)$. Moreover,
$$d(x,I)\leqslant\frac{1}{\epsilon}P(x)\ \ \ \ \ \ \ \ \forall x.$$
So, $I$ is simply definable. This is a contradiction.
\end{proof}

The following is another variant of overspill which holds with some additional requirements.

\begin{proposition} (Overspill)
Let $D\subseteq M$ be simply definable and $I$ be a proper cut with $I<I^c$.
Assume for every $a\in I$ there is $a\leqslant u\in I\cap D$.
Then, for every $\epsilon>0$ and $b\in I^c$ there is $b\geqslant v\in I^c\cap D^\epsilon$.
\end{proposition}
\begin{proof}
Assume not. Then, there are $0<\epsilon<1$ and $b\in I^c$ for which there is no $v$ with
$$b\geqslant v\in I^c\cap D^\epsilon.$$ Hence, $[0,b]\subseteq I\cup(D^\epsilon)^c$.
Let
$$P(x)=\inf_y d((x\vee y)\wedge b, D).$$ $P$ is a formula.
Moreover, by the assumptions, $I\subseteq Z(P)$. Conversely, suppose that $x\in I^c$.
Then, for every $y$ $$b\geqslant(x\vee y)\wedge b=(x\wedge b)\vee(y\wedge b)\in I^c$$
which shows that $(x\vee y)\wedge b\in(D^\epsilon)^c$. Hence, $\epsilon\leqslant P(x)$.
We conclude that $Z(P)=I$. Furthermore,
$$d(x,Z(P))\leqslant\frac{1}{\epsilon}P(x)\ \ \ \ \ \ \ \ \forall x.$$
So, $I$ is simply definable. This is a contradiction.
\end{proof}

\begin{proposition} (Underspill)
Let $I$ be a proper cut and $D\subseteq M$ be simply definable. 

(i) Assume $1\leqslant I^c\subseteq D$. Then, for each $\epsilon>0$ there exists $a\in I$ such that $[a,\infty)\subseteq D^\epsilon$.

(ii) Assume $I<I^c$. If for every $a\in I^c$ there is $a\geqslant u\in I^c\cap D$,
then every $\epsilon>0$ and $b\in I$ there is $b\leqslant v\in I\cap D^\epsilon$.
\end{proposition}
\begin{proof}
(i) Assume not and set $$P(x)=\sup_{x\leqslant y}d(y,D)$$ 
Then, $P(x)$ is formula and $I^c=Z(P)$. Moreover,
$$d(x,Z(P))\leqslant\frac{1}{\epsilon}P(x)\ \ \ \ \ \ \ \ \forall x.$$
Hence, $I^c$ is simply definable. This is a contradiction.

(ii): Assume not. Then, there are $\epsilon>0$ and $b\in I$ such that $[b,\infty)\cap I\cap D^\epsilon=\emptyset$.
Set $$P(x)=\inf_yd((x\wedge y)\vee b,D)$$  
Then $P$ is a formula and $I^c\subseteq Z(P)$. Conversely, if $x\in I$, then for every $y$
$$b\leqslant(x\wedge y)\vee b=(x\vee b)\wedge(y\vee b)\in I$$
which shows that $(x\wedge y)\vee b\not\in D^\epsilon$. Hence, $d((x\wedge y)\vee b, D)\geqslant\epsilon$.
We conclude that $I^c=Z(P)$ and $$d(x,Z(P))\leqslant\frac{1}{\epsilon}P(x)\ \ \ \ \ \ \ \ \forall x.$$
This is a contradiction since $I^c$ is not simply definable.
\end{proof}

\section{Division and coding}
In this section, we will examine some number theoretic properties in affine models.
As before, to prove a claim in AA, we may show that it is an affine consequence of PA.
We assume that $M\vDash\textrm{AA}$ is
$\aleph_0$-saturated. However, as stated before, as long as uniqueness is concerned, saturation is unnecessary.
An easier way is to verify such claims in every model of the form $\int^\oplus_\Omega M_\omega d\mu(\omega)$
where $M_\omega$ is a model of PA. Both methods will be used.
So, for example, to see that $1$ is the only invertible element in models of AA,
we verify that the following condition holds in every model of PA: $$d(x,1)+d(y,1)\leqslant2d(1,xy).$$

The Euclidean division algorithm in classical models can be restated by the affine condition
$$\inf_{uv}[d(x,yu+v)+d(y\wedge(v+1),v+1)]\leqslant1-|y|\ \ \ \ \ \ \ \ \ \forall xy.$$
It states that if $y\neq0$, there are $u,v$ such that $x=yu+v$ and $v<y$.
So, correspondingly, it holds in models of affine arithmetic and states that
if $y$ is normal, there are $u,v$ such that $x=yu+v$ and $v+1\leqslant y$ (or that $v\leqslant y-1$).
It can be proved that $u,v$ are unique and definable from $x,y$.

More generally, every provably total uniformly first-order definable function of PA gives an affinely definable function
of AA interpreted pointwise in M. In particular, if $\gamma(\x, y)$ defines the graph of $f$, then (in every
model of PA) one has $d(f(\x),y)=1-\gamma(\x,y)$.
Affine reduction makes this distance predicate affinely definable.
For example, if $y$ is normal, then $y(\omega)\neq0$ almost everywhere. Define $q, r\in M$ by ordinary Euclidean
division in each fiber by setting:
$$x(\omega)=y(\omega)q(\omega)+r(\omega),\ \ \ \ \ \ \ \ r(\omega)<y(\omega).$$
The unique quotient and remainder are uniformly PA-definable functions, so $q$ and $r$ are measurable
sections and are affinely definable. We have therefore, the following proposition.

\begin{proposition}
There are definable functions $f$, $g$ on the set of pairs $(x,y)$ where $y$ is normal
such that for every such pair,\ \ $x=y\cdot f(x,y)+g(x,y)$ and $g(x,y)+1\leqslant y$.
\end{proposition}

The remainder function ($g(x,y)$ above) is usually denoted by $\big(\frac{x}{y}\big)$.
Coprimality is another number theoretic notion. One verifies easily that nonzero elements
$a,b\in M\vDash\textrm{PA}$ are \emph{coprime} if and only if
$$M\vDash\ d(z,1)\leqslant\inf_wd(wz,a)+\inf_wd(wz,b)\ \ \ \ \ \ \ \ \forall z\in M.$$
So, we can take it as a definition of coprimality in models of AA for normal elements $a,b$.
As usual, if $m$ is normal, $x\equiv y$ (mod $m$) means that $\big(\frac{x}{m}\big)=\big(\frac{y}{m}\big)$.

The following results can be proved by using direct integrals in a similar way.

\begin{lemma} (B\'{e}zout's theorem)
For any normal $x,y\in M\vDash\textrm{AA}$, if $(x,y)=1$, then there is $z\leqslant y-1$ such that $zx\equiv1$ (mod $y$).
\end{lemma}

\begin{lemma} (Chinese remainder theorem)
Let $M$ be a model of AA and $m_0,...,m_{n-1}$ be normal and pairwise coprime.
Then, for every $x_0,...,x_{n-1}$ there is $y$ such that
$$y\equiv x_i\ \ (\mbox{mod}\ m_i)\ \ \ \ \ \ \ \ i=0,...,n-1.$$
\end{lemma}


By definition, $p\in M\vDash\textrm{PA}$ is irreducible (resp. prime) if $p\geqslant2$ and for every $x,y$,\ \
$p=xy$ implies that $p=x$ or $p=y$ (resp. $p|xy$ implies that $p|x$ or $p|y$).
These notions are equivalent in models of PA. Let $$\textrm{irred}(p)=\sup_{xy}[d(p,x)+d(p,y)-d(p,xy)]-1$$
$$\textrm{prime}(p)=\sup_{xy}\ [\inf_{t}d(pt,x)+\inf_{t}d(pt,y)-\inf_{t}d(pt,xy)]-1$$
Then, for any $2\leqslant p\in M\vDash\textrm{PA}$ one has that $\textrm{irred}(p)=0$ if and only if $p$ is irreducible.
Otherwise, $\textrm{irred}(p)=1$. Similarly, for any $2\leqslant p\in M\vDash\textrm{PA}$ one has that
$\textrm{prime}(p)=0$ if and only if $p$ is prime. Otherwise, $\textrm{prime}(p)=1$.
Therefore, correspondingly, for any $2\leqslant p\in M\vDash\textrm{AA}$ we have that
$0\leqslant\textrm{irred}(p)$ and $0\leqslant\textrm{prime}(p)$.
We then say $p\geqslant 2$ is \emph{irreducible} if $\textrm{irred}(p)=0$ and it is
\emph{prime} if $\textrm{prime}(p)=0$.

\begin{lemma}
$2\leqslant p\in M\vDash\textrm{AA}$ is irreducible if and only if it is prime.
\end{lemma}
\begin{proof}
By the above explanation, if $M\vDash\textrm{PA}$, then $M$ satisfies
$$\textrm{irred}(x)=\textrm{prime}(x)\hspace{15mm} \forall x\geqslant2.$$
We conclude that this condition belongs to AA.
Hence, being prime or irreducible in models of AA are equivalent.
\end{proof}

Note that the classical definitions do not coincide here. For example, it is not hard to verify
that $p=(2,2)\in\frac{1}{2}\Nn+\frac{1}{2}\Nn$ is irreducible in the affine sense.
However, one has that $$p=(2,1)\cdot(1,2).$$

As stated before, if $K$ has the end-set property and $K\preccurlyeq M$ is extremally
$\aleph_0$-saturated, then $M$ has the end-set property.
We conclude that if $M\vDash\textrm{AA}$ is saturated and has a first-order elementary substructure,
then prime numbers form a definable set in $M$.
This raises the question of whether every saturated $M\vDash\textrm{AA}$ has a first-order elementary substructure.
By Corollary \ref{prime model}, this is true if $\textrm{dcl}_M(\emptyset)$ is first-order.

\begin{lemma}
If $2\leqslant x$, then there is an irreducible $p$ such that $p|x$.
\end{lemma}
\begin{proof}
As stated above we have always that $\textrm{irred}(p)\leqslant1$.
It is then easily verified that in models of PA the condition
$$\inf_{2\leqslant p}\ [\textrm{irred}(p)+\inf_td(pt,x)]\leqslant2d(2\wedge x,2)$$
holds for $x\geqslant2$ if and only if there is an irreducible $p$ which divides $x$.
The claim then follows from the fact that this condition holds in models of AA too.
\end{proof}

The unboundedness of primes can be proved by an affinization of the classical proof.
It is however easier to use the following affine condition which is satisfied in every model of PA:
$$\sup_x\inf_y[d(x\wedge y,x)+\textrm{prime}(y)]\leqslant0.$$
Applying this for $M$, we conclude that for every $x\in M$ there is a prime $p\geqslant x$.

The irrationality of $\sqrt 2$ is also proved in AA regarding the fact that PA satisfies
the condition $$d(x,y)\leqslant d(x^2,2y^2).$$
In fact, this implies that if $x^2=2y^2$, then $x=y=0$.

Although it is not true that every element in $M$ is either even or odd,
it is  true that $x(x+1)$ is always even. This is because PA proves the condition
$$\inf_u d(2u,x(x+1))=0.$$
In particular, there is a unique $u$ with $2u=x(x+1)$ and the function $x\mapsto\frac{1}{2}x(x+1)$
is definable. 
We can therefore set
$$\langle x,y\rangle=\frac{(x+y+1)(x+y)}{2}+y$$
which is again a definable function. 
Moreover, the affine condition $$\sup_z\inf_{xy}d(\langle x,y\rangle,z)\leqslant0$$
shows that the function $\langle x,y\rangle$ is surjective.
It is also injective since the affine condition
$$d(u,x)+d(v,y)\leqslant 2d(\langle u,v\rangle,\langle x,y\rangle)$$
holds in models of PA. Note also that these facts are expressible by conditions
of the form $\phi\leqslant0$ where $\phi$ is $\Sigma_2$.
This will be used later in the proof of Theorem \ref{Gaifman1}.
Since $\langle x,y\rangle$ is bijective, its inverse $\langle x,y\rangle\mapsto(x,y)$ is definable.
Now, if $x=\langle a,m\rangle$, set $$(x)_i=\Big(\frac{a}{m(i+1)+1}\Big).$$
Since the remainder function is definable, $(x)_i$ is a definable function of the pair $(x,i)$.
Anyway, pairing function, $\beta$-coding function $(x)_i$ etc. are defined in the framework of AA.
All these can be also verified using the direct integral representation (as they are uniformly PA-definable).
We show that $(x)_i$ has the expected properties. Let us recall the classical result in this respect.

\begin{lemma}
PA satisfies the following sentences:

(i) $\forall x\exists y (y)_0=x$

(ii) $\forall xyz\exists w[\forall i<z((w)_i=(y)_i)\wedge(w)_z=x]$

(iii) $\forall xy\ (x)_y\leqslant x$.
\end{lemma}

Since $<$ is not expressible here, we must make minor changes in (ii).
We have then the following lemma.

\begin{lemma} \label{coding lemma} AA satisfies the following conditions:

(i) $\sup_x\inf_y d((y)_0,x)\leqslant0$

(ii) $\sup_{xy}\sup_{1\leqslant z}\inf_w[\sup_{i\leqslant z-1}d((w)_i,(y)_i)+d((w)_z,x)]\leqslant0$

(iii) $\sup_{xy}d((x)_y\vee x,x)=0$
\end{lemma}
\begin{proof}
All these conditions hold in models of PA.
\end{proof}

We can rewrite Lemma \ref{coding lemma} as follows:

(i) $\forall x\exists y (y)_0=x$

(ii) $\forall xyz[1\leqslant z\rightarrow\exists w(\forall i\leqslant z-1((w)_i=(y)_i)\wedge(w)_z=x)]$

(iii) $\forall xy\ (x)_y\leqslant x$.
\bigskip

As in the classical case, (using affine reduction) the relation $(x)_y=z$ is stated by an affine condition of the form
$\theta=0$ where $\theta$ is a $\Delta_0$-formula (see \cite{Kaye} p. 64).

\begin{lemma} (G\"{o}del's lemma)
Let $x_0,...,x_{n-1}\in M$. Then, there exists $u\in M$ such that
$$M\vDash(u)_i=x_i\ \ \ \ \ \ \ \ \ \mbox{for}\ i=0,...,n-1.$$
\end{lemma}
\begin{proof}
This can be proved either by the Chinese remainder theorem or using the following affine condition
which holds in all models of PA:
$$\sup_{x_0...x_{n-1}}\inf_u[d((u)_0,x_0)+\cdots+d((u)_{n-1},x_{n-1})]\leqslant0.$$
\end{proof}

The usual functions such as factorial, exponential etc can be defined either by using Lemma \ref{coding lemma}
or by using direct integrals. More important is the following lemma where by uniformly definable is meant
definable by the affine reduction of a provably total PA-definable function.


\begin{lemma} \label{definable coding}
Let $\phi(x,y)$ be a formula with parameters in $M$.
Assume there is a uniformly definable function $f(x)$ such that
$0\leqslant\phi^M(x,f(x))$ for all $x$. Then, $$0\leqslant\inf_x\sup_w\inf_{i\leqslant x}\phi^M(i,(w)_i).$$
\end{lemma}
\begin{proof}
Let $M=\int_\Omega^\oplus M_\omega d\mu(\omega)$. Given a bound $b\in M$, choose in each fiber the least
code $w(\omega)$ of the finite sequence $f^{M_\omega}(0),..., f^{M_\omega}(b(\omega))$ and form
the required code satisfying the above formula.
\end{proof}

\section{The splitting theorem}
For $\epsilon\geqslant0$, we write $x\leqslant_\epsilon y$ if $d(x\wedge y,x)\leqslant\epsilon$.
So, $x\leqslant_0y$ coincides with $x\leqslant y$.
Let $M\subseteq N$ be models of $\textrm{AA}^\sim$. Then, $N$ is an \emph{approximately cofinal extension}
of $M$ if for every $x\in N$ and $\epsilon>0$ there is $y\in M$ such that $x\leqslant_\epsilon y$.
The exact form of cofinal extension is defined in the obvious way.
It is an \emph{end extension} of $M$ if for every $x\in N$,
\ \ $x\leqslant y\in M$ implies that $x\in M$. These are respectively denoted by
$M\subseteq_{\textrm{acf}}N$, $M\subseteq_{\textrm{cf}} N$ and $M\subseteq_{\textrm{e}}N$. Also, one writes $M\preccurlyeq_{\Sigma_n}N$
if $\phi^M(\a)=\phi^N(\a)$ for every $\Sigma_n$-formula $\phi(\x)$ and $\a\in M$.
In Example \ref{examples-bounded},
one has that $\Nn\subseteq_{\textrm{acf}}\bar{\mathbf{N}}\subseteq_{\textrm{e}}\mathbb{N}^\mu$.
The following lemma is an easy consequence of the definitions.

\begin{lemma}\label{end-sigma0}
Assume $M\subseteq_{\textrm{e}} N$ are models of $\textrm{AA}^\sim$. Then, $M\preccurlyeq_{\Sigma_0}N$.
\end{lemma}

\begin{proposition}
Assume $M\subseteq_{\textrm{acf}} N$ are models of $\textrm{AA}^\sim$ and $M\preccurlyeq_{\Sigma_0}N$.
If $M$, $N$ satisfy the collection axioms, then $M\preccurlyeq N$.
\end{proposition}
\begin{proof}
The proof is similar to the classical case. $M\preccurlyeq_{\Sigma_0}N$ holds by the assumptions.
We prove that for each $n$, $M\preccurlyeq_{\Sigma_n}N$ implies $M\preccurlyeq_{\Sigma_{n+1}}N$.
For this purpose, it is sufficient (as the converse direction is obvious) to assume
$M\preccurlyeq_{\Sigma_n}N$ and prove that for every $\phi(\x,\y)\in\Pi_n$ and $\a\in M$
$$r\leqslant\sup_{\y}\phi^N(\a,\y)\ \ \ \Longrightarrow\ \ \ r\leqslant\sup_{\y}\phi^M(\a,\y)\ \ \ \ \ \ \ \ \forall r.$$
So, suppose that the left-hand side of the asserted implication holds. Let $\epsilon>0$ be given.
Then, by the approximate cofinality assumption and continuity, there is $c\in M$
such that $$r-\epsilon\leqslant\sup_{\y\leqslant c}\phi^N(\a,\y).$$
By Lemma \ref{coll}, $\sup_{\y\leqslant c}\phi(\a,\y)\in\Pi_n({\textrm{Coll}})$. So, by the induction hypothesis,
we must have that $$r-\epsilon\leqslant\sup_{\y\leqslant c}\phi^M(\a,\y).$$
So, $$r-\epsilon\leqslant\sup_{\y}\phi^M(\a,\y)$$ and the claim is proved as $\epsilon$ is arbitrary.
\end{proof}

\begin{proposition}\label{Gaifman1}
Let $M\subseteq_{\textrm{acf}}N$ be models of $\textrm{AA}^\sim$. Assume $M\preccurlyeq_{\Sigma_0}N$
and $M\vDash\textrm{AA}$. Then, $M\preccurlyeq N$ and hence $N\vDash\textrm{AA}$.
\end{proposition}
\begin{proof}
Let $\mathcal F$ be a non-principal ultrafilter on $\Nn$.
Then, $M^{\mathcal F}\subseteq_{\textrm{cf}}N^{\mathcal F}$ and they are both extremally $\aleph_1$-saturated.
Moreover, $M^{\mathcal F}\preccurlyeq_{\Delta_0}N^{\mathcal F}$.
So, we may return to the original structure and assume without loss that $M$ and $N$ are extremally $\aleph_1$-saturated.
We first show that for every $\Sigma_2$-formula $\phi(\x)$,\ \ $\a\in M$ and $r$,
 $$r\leqslant\phi^N(\a)\ \ \ \Longrightarrow\ \ \ r\leqslant\phi^M(\a).$$
Assume $\phi(\x)=\sup_{\y}\inf_{\z}\psi(\x,\y,\z)$ where $\psi\in\Sigma_0$. Let $\epsilon>0$.
Then, by the assumptions, there is $b\in M$ such that for all $c\in M$ one has that
$$r-\epsilon\leqslant\sup_{\y\leqslant b}\inf_{\z\leqslant c}\psi^N(\a,\y,\z).$$
Since $M\preccurlyeq_{\Sigma_0}N$, we conclude that
$$r-\epsilon\leqslant\inf_{w}\sup_{\y\leqslant b}\inf_{\z\leqslant w}\psi^M(\a,\y,\z)$$
and hence by collection axiom (indeed, its contraposition form)
$$r-\epsilon\leqslant\sup_{\y\leqslant b}\inf_{\z}\psi^M(\a,\y,\z)\leqslant\phi^M(\a).$$
So, since $\epsilon>0$ is arbitrary, $r\leqslant\phi^M(\a)$.
Now, for each $n\geqslant2$, we assume $M\preccurlyeq_{\Sigma_n}N$ and prove that $M\preccurlyeq_{\Sigma_{n+1}}N$.
Equivalently, that for every $\psi(\x)\in\Pi_{n+1}$ and $r$
$$r\leqslant\psi^M(\a)\ \ \ \Longrightarrow\ \ \ r\leqslant\psi^N(\a).$$
As mentioned before, there are conditions of the form $\phi\leqslant0$, where $\phi$ is $\Sigma_2$,
stating that the pairing function is definable and bijective.
These conditions hold in $M$ as $M\vDash\textrm{AA}$. So, they hold in $N$ too.
Consequently, we may assume $\y$ and $\z$ are single variables.

Let $\psi(\a)=\inf_y\sup_z\phi(\a,y,z)$ be a formula with $\phi\in \Pi_{n-1}$.
By the approximate cofinality assumption, we only need to prove that for each $b\in M$:
$$r\leqslant\inf_{y\leqslant b}\sup_z\phi^M(\a,y,z)\ \ \ \Longrightarrow\ \ \ r\leqslant\inf_{y\leqslant b}\sup_z\phi^N(\a,y,z).$$
So, assume the left-hand side of the above implication holds.
Then, by collection principle and extremal saturation assumption, there is $s$ such that
$$r\leqslant\inf_{y\leqslant b}\sup_{z\leqslant s}\phi^M(\a,y,z).$$
By LNP, there is a (uniformly) definable function $f$ such that
$$r\leqslant\inf_{y\leqslant b}\phi^M(\a,y,f(y)).$$
By Corollary \ref{definable coding},
$$r\leqslant\inf_{y\leqslant b}\sup_w\inf_{i\leqslant y}\phi^M(\a,i,(w)_i).$$
Therefore, by setting $y=b$ and using extremal saturation (then replacing $i$ with $y$), we find $w\in M$ such that
$$r\leqslant\inf_{y\leqslant b}\phi^M(\a,y,(w)_y).$$
Since $(w)_y=z$ is stated by an affine $\Delta_0$ condition, if $(w)_y=z$ holds in $M$, it holds in $N$ too.
So, since $\{y\in M:\ y\leqslant b\}$ is dense in $\{y\in N:\ y\leqslant b\}$,
we have by the induction hypothesis that
$$r\leqslant\inf_{y\leqslant b}\phi^N(\a,y,(w)_y)$$
which clearly implies the intended claim.
\end{proof}

\begin{lemma}\label{Lemma for MRDP}
Let $\phi(\x)$ be a bounded affine formula. Then, for every $r$ there is a first-order bounded $\theta(\x)$
such that for every $\a\in M\vDash\textrm{PA}$,\ \ $M\vDash r\leqslant\phi(\a)$ iff $M\vDash\theta(\a)$.
\end{lemma}
\begin{proof}
We may assume $\phi$ prenex, i.e. of the form $Q_n\cdots Q_1\psi(\z)$ where $\psi(\z)$ is
quantifier-free and every $Q_i$ is a bounded quantifier of the form $\sup_{\y\leqslant t}$ or $\inf_{\y\leqslant t}$.
Assume the range of $\psi^M(\z)$ is included in the set $\{r_1,...,r_k\}$ for any $M\vDash\textrm{PA}$.
It is easy to see that for each $r$ there is a quantifier-free first-order $\theta(\z)$ such that
$M\vDash r\leqslant\psi(\a)$ iff $M\vDash\theta(\a)$ whenever $\a\in M\vDash\textrm{PA}$.
So, correspondingly, $r\leqslant\sup_{\y\leqslant t}\psi$ is equivalent to $\exists\y\leqslant t\ \theta$
and $r\leqslant\inf_{\y\leqslant t}\psi$ is equivalent to $\forall\y\leqslant t\ \theta$ in any $M\vDash\textrm{PA}$.
Now, use induction on $n$.
\end{proof}


The classical MRDP theorem states that every r.e. set is Diophantine.
In the logical language, it states that every (first-order) $\Sigma_1$-formula is PA-equivalent to
an existential formula. Equivalently, every $\Pi_1$-formula is PA-equivalent to a universal formula.
This theorem is provable in PA.
By a $\sup$-formula (resp. $\inf$-formula) we mean an affine formula of the form $\sup_{\y}\psi$
(resp. $\inf_{\y}\psi$) where $\psi$ is affine quantifier-free.
The following proposition is a form of MRDP theorem in affine arithmetic.

\begin{proposition} \label{WMRDP}
For every affine bounded formula $\phi(\x)$ there are affine $\sup$-formula $\xi(\x)$
and affine $\inf$-formula $\eta(\x)$ such that
$$AA\ \vDash\ \phi(\x)=\xi(\x)=\eta(\x).$$
\end{proposition}
\begin{proof}
Let $\psi(\z)$ and $\{r_1,...,r_k\}$ be as in Lemma \ref{Lemma for MRDP}.
There is no harm if we further assume that $0<r_1<\cdots<r_k$.
The lemma implies that any condition $\phi(\x)=r_i$
is stated by a bounded first-order $\theta_i(\x)$, i.e. $\phi^M(\a)=r_i$
iff $M\vDash\theta_i(\a)$ for any $\a\in M\vDash\textrm{PA}$.
By the MRDP theorem, for each $i$ there are existential $\xi_i(\x)$
and universal $\eta_i(\x)$ such that
$$\textrm{PA}\vDash\theta_i(\x)\leftrightarrow\xi_i(\x)\leftrightarrow\eta_i(\x).$$
It was proved in Lemma \ref{affine reduction} that every quantifier-free first-order formula
is PA-equivalent to a quantifier-free affine formula.
Correspondingly, every existential (resp. universal) first-order formula is PA-equivalent
to a $\sup$-formula (resp. $\inf$-formula).
So, we may now assume that every $\xi_i$ is a $\sup$-formula and every $\eta_i$ is an $\inf$-formula
and write $$\textrm{PA}\vDash\theta_i(\x)=\xi_i(\x)=\eta_i(\x).$$
Let $$\xi(\x)=\sum_ir_i\xi_i(\x), \hspace{15mm} \eta(\x)=\sum_ir_i\eta_i(\x).$$
Then, $\xi$ is equivalent to a $\sup$-formula and $\eta$ is equivalent to an $\inf$-formula.
Moreover, since $\theta_i$'s are pairwise inconsistent,
$$\textrm{PA}\vDash\phi(\x)=\xi(\x)=\eta(\x).$$
We conclude that these conditions are satisfied in AA too.
\end{proof}

One shows by Proposition \ref{WMRDP} that for models $M,N$ of AA, $M\subseteq N$ implies $M\preccurlyeq_{\Sigma_0}N$.

\begin{corollary} \label{Gaifman2} (Gaifman)
Let $M\subseteq_{\textrm{acf}}N$ be models of $\textrm{AA}^\sim$ and $M\vDash\textrm{AA}$.
Then, the following are equivalent:

(i) $M\preccurlyeq_{\Sigma_0}N$

(ii) $N\vDash\textrm{AA}$

(iii) $M\preccurlyeq N$.
\end{corollary}

\begin{corollary} \label{affine splitting}(Affine splitting theorem)
If $M\subseteq N$ are models of AA, then there is a unique $K\vDash\textrm{AA}$ such that
$M\subseteq_{\textrm{acf}}K\subseteq_{\textrm{e}}N$ and $M\preccurlyeq K$.
\end{corollary}
\begin{proof}
Let $K$ be the topological closure of the set $$K_0=\{a\in N:\ \exists b\in M\ \ a\leqslant b\}.$$
Then, $K$ is a model of $\textrm{AA}^\sim$ and $M\subseteq_{\textrm{acf}}K\subseteq_{\textrm{e}}N$.
By Lemma \ref{end-sigma0}, $K\preccurlyeq_{\Sigma_0}N$. Also, by Proposition \ref{WMRDP},
$M\preccurlyeq_{\Sigma_0}N$. We conclude that $M\preccurlyeq_{\Sigma_0}K$ and hence
by Proposition \ref{Gaifman1}, $M\preccurlyeq K$. The uniqueness of $K$ is obvious.
\end{proof}

In Peano arithmetic, one uses the fact that $\Nn\subseteq_{\textrm{e}}M$ to show that
the $\Sigma_1$-theory of $\Nn$ holds in every model of PA. This can be similarly done here.
By the (affine) $\Sigma_1$-theory of $\Nn$ is meant
$$\{0\leqslant\sup_{\x}\phi(\x):\ \phi\in\Sigma_0,\ 0\leqslant\sup_{\x}\phi^\Nn(\x)\}.$$
Let $M\vDash\textrm{AA}$. We have generally that $\Nn\subseteq M$.
Let $$K=\{a\in M:\ \exists n\in\Nn\ \ a\leqslant n\}.$$
By Corollary \ref{affine splitting}, $\Nn\subseteq_{\textrm{acf}}K\subseteq_{\textrm{e}}M$
and $\Nn\preccurlyeq K\preccurlyeq_{\Sigma_0}M$.
Now, for every $\Sigma_1$-sentence $\sup_{\x}\phi(\x)$,
if $0\leqslant\sup_{\x}\phi^{\Nn}(\x)$, then $0\leqslant\sup_{\x}\phi^M(\x)$.

\begin{proposition} \label{cofinal extension}
Let $M\vDash\textrm{AA}$ contain a salient element, i.e. $b\in M$ such that
for some $\epsilon>0$, one has that $b\nleqslant_\epsilon n$ for all $n$.
Then, $M$ has a proper approximately cofinal elementary extension.
\end{proposition}
\begin{proof}
By the assumption, $\epsilon\leqslant d(n\wedge b,b)$ for all $n$. Let $M=\int_\Omega^\oplus M_\omega d\mu(\omega)$ where $M_\omega\vDash$ PA.
For each $n$, let $X_n\subseteq\Omega$ be the set of $\omega\in\Omega$ such that $n\wedge b_\omega\neq b_\omega$.
Then, $\mu(X_n)\geqslant\epsilon$, $X_n$ is decreasing and $n<b_\omega$ whenever $\omega\in X_n$.
Let $X=\cap_nX_n$. Then, $\mu(X)\geqslant\epsilon$ and for each $n$,\ \ $n<b_\omega$ whenever $\omega\in X$.
Let $c$ be a new constant symbol and $$T=ediag(M)\cup\{\frac{\epsilon}{3}\leqslant d(a,c):\ a\in M\}\cup\{d(c\wedge b,c)=0\}$$
where $ediag(M)$ is the affine elementary diagram of $M$.
We claim that for every $a_1,...,a_k$ in $M$, the condition $\frac{\epsilon}{3}\leqslant\sum_ir_id(a_i,c)$
is satisfied in $M$ by some suitable $c=n\wedge b$ where $\sum_ir_i=1$.
Assume not. Then there are $a_1,...,a_k$ such that for every $n$
$$\sum_ir_id(a_i,n\wedge b)\leqslant\frac{\epsilon}{3}.$$
So, for each $n$ there is $1\leqslant\ell\leqslant k$ such that $d(a_\ell,n\wedge b)\leqslant\frac{\epsilon}{3}$.
Indeed, this happens for some fixed $\ell$ and every $n$ in some infinite $X\subseteq\Nn$.
Therefore, for $m,n\in X$ where $m<n$,
$$\epsilon\leqslant d(0,n\wedge b\dotminus m\wedge b)=d(m\wedge b,n\wedge b)\leqslant d(m\wedge b,a_\ell)+d(a_\ell,n\wedge b)\leqslant\frac{2\epsilon}{3}.$$
which is a contradiction. We conclude by affine compactness theorem that $T$ has a model say $N$.
Let $K$ be the topological closure of the set $$K_0=\{d\in N:\ d\leqslant a\ \ \textrm{for some}\ a\in M\}$$
in $N$. Then, by Corollary \ref{affine splitting}, $M\preccurlyeq N$ and $M\subseteq_{\textrm{acf}}K\subseteq_e N$
and $M\preccurlyeq K$. Also, $c\in K-M$.
\end{proof}

\begin{corollary}
Every model $M$ of AA has arbitrarily large approximately cofinal elementary extensions.
\end{corollary}
\begin{proof}
By Example \ref{examples-bounded}, $\Nn$ has arbitrarily large approximately cofinal elementary extensions.
The same is true if $\Nn\subseteq_{\textrm{acf}}M$ since, in this case, $M$ satisfies the strong form of
collection axioms and a similar argument applies.
On the other hand, by Proposition \ref{cofinal extension}, this is also true if
$M\vDash\textrm{AA}$ has a salient element.
\end{proof}

\noindent{\bf Question:} Find a complete axiomatization for AA and reprove results on its basis.

\noindent{\bf Question:} Find the logical relations between the induction axioms, LNP and Coll.
\bigskip

\noindent{\bf Acknowledgements:} The author thanks an anonymous referee for important comments and suggestions
that helped to improve significantly the paper.

\end{document}